\documentclass[fleqn,11pt,oneside]{article}

\usepackage{amsmath}
\usepackage{amsthm}
\usepackage{amsfonts}
\usepackage{amssymb}
\newtheorem{theorem}{Theorem}
\usepackage{graphicx}
\usepackage{subcaption}
\usepackage{physics}
\usepackage{relsize}
\usepackage{verbatim}
\usepackage{xifthen}
\usepackage{xfrac}
\usepackage{xargs}
\usepackage{xcolor}
\usepackage{algorithm,algpseudocode}
\usepackage{caption,subcaption}
\usepackage{booktabs,multirow}
\usepackage{setspace} 
\usepackage[top=1.1in, bottom=1.1in, left=1.6in, right=1.1in]{geometry}
\usepackage[final]{showlabels}
\usepackage{dashbox}
\usepackage{microtype}
\usepackage[medium]{titlesec}
\usepackage{algpseudocode}
\usepackage{xpatch}
\usepackage{accents}
\usepackage{bbding}
\usepackage[hidelinks]{hyperref}

\newtheorem{lemma}[theorem]{Lemma}

\newtheorem{definition1}{Definition}[section]
\newtheorem{observe}{Observation}[section]
\newtheorem{remark1}[observe]{Remark}
\newtheorem{example1}{Example}[section]
\newtheorem{aside1}[observe]{Aside}

\newenvironment{remark}{\begin{remark1} \rm}{\end{remark1}}

\algdef{SE}[DOWHILE]{Do}{doWhile}{\algorithmicdo}[1]{\algorithmicwhile\ #1}%

\def\C{\mathbb{C}}

\def\1{\mathbbm{1}}

\newcommand \I{\mathrm{i}}
\newcommandx\D[1][1=]{\mathrm{d}#1}
\newcommand {\Exp}[1]{\mathrm{e}^{#1}}

\newcommandx {\equals}[3][1={=}]{\underbrace{#2}_{#1#3}}
\newcommand{\mat}[1]{\begin{matrix}#1\end{matrix}} 
 
\newcommand{\bb}[1]{\left(#1\right)} 
\renewcommand{\sb}[1]{\left[#1\right]} 

\newcommand{\cb}[1]{\left\{#1\right\}}

\renewcommand{\eqref}[1]{(\ref{eq:#1})}

\renewcommand{\hat}{\widehat}
\newcommand{\eqdef}{\mathrel{\mathop:}=}

\newcommand{\hl }[1]{{\color{black} #1}}

\newcommand\restr[2]{{
  \left.\kern-\nulldelimiterspace 
  #1 
  \littletaller 
  \right|_{#2} 
  }}

\newcommand{\littletaller}{\mathchoice{\vphantom{\big|}}{}{}{}}

\def\qed{\hfill$\blacksquare$\\} \renewenvironment{proof}{\noindent {\bf 
Proof.}}{\qed}

\begin{document}

%
%
%
%

\begin{center}
   \begin{minipage}[t]{6.0in}
	We present a scheme for finding all roots of an analytic function in a square domain in the complex plane. The scheme can be viewed as a generalization of the classical approach to finding roots of a function on the real line, by first approximating it by a polynomial in the Chebyshev basis, followed by diagonalizing the so-called ``colleague matrices''. Our extension of the classical approach is based on several observations that enable the construction of polynomial bases in compact domains that satisfy three-term recurrences and are reasonably well-conditioned. This class of polynomial bases gives rise to ``generalized colleague matrices'', whose eigenvalues are roots of functions expressed in these bases. In this paper, we also introduce a special-purpose QR algorithm for finding the eigenvalues of generalized colleague matrices, which is a straightforward extension of the  recently introduced  structured stable QR algorithm for the classical cases (See \cite{serkh2021}). The performance of the schemes is illustrated with several numerical examples.
	\thispagestyle{empty}

  \vspace{ -100.0in}

  \end{minipage}
\end{center}


\vspace{ 2.60in}
\vspace{ 0.50in}

\begin{center}
  \begin{minipage}[t]{4.4in}
    \begin{center}

\textbf{Finding roots of complex analytic functions via generalized colleague matrices} \\

  \vspace{ 0.50in}

H. Zhang$\mbox{}^{\dagger\, \diamond}$,
V. Rokhlin$\mbox{}^{\dagger \, \,\oplus}$  \\
              \today

    \end{center}
  \vspace{ -100.0in}
  \end{minipage}
\end{center}

\vspace{ 2.00in}

\vfill

\noindent 
$\mbox{}^{\diamond}$  This author's work was supported in part under ONR N00014-18-1-2353. \\
$\mbox{}^{\oplus}$  This author's work was supported in part under ONR N00014-18-1-2353 and NSF DMS-1952751. \\
\vspace{1mm}

\noindent
$\mbox{}^{\dagger}$ Department of Mathematics, Yale University, New Haven, CT 06511 \\
\vspace{1mm}

\vfill
\noindent
{\footnotesize Code available at \href{https://github.com/han-wen-zhang/colleague_rootfinders}{https://github.com/han-wen-zhang/colleague\_rootfinders}}

\vspace{2mm}


\vfill
\eject


\tableofcontents

\section{Introduction}
In this paper, we introduce a numerical procedure for finding all roots of an analytic function $f$ over a square domain in the complex plane. The method first approximates the function to a pre-determined accuracy by a polynomial 
\begin{equation}
	p(z) = c_0 + c_1 P_1(z) + c_2 P_2(z) + \cdots + c_n P_n(z)\,,
\end{equation}
where $P_0, P_1, \ldots, P_n,\ldots$ is a polynomial basis. For this purpose, we introduce a class of bases that (while not
orthonormal in the classical sense) are reasonably well-conditioned,
and are defined by three-term recurrences of the form 
\begin{equation}
	zP_j(z) = \beta_j P_{j-1}(z) + \alpha_{j+1}P_{j}(z) + \beta_{j+1} P_{j+1}(z)\,,
\end{equation}
given by complex coefficients $\alpha_1, \alpha_2, \ldots, \alpha_n$ and $\beta_1, \beta_2, \ldots, \beta_n$.
Subsequently, we calculate the roots of the obtained polynomial
by finding the eigenvalues of a ``generalized colleague
matrix'' via a straightforward generalization of the scheme
of \cite{serkh2021}. The resulting algorithm is a complex orthogonal QR algorithm that requires order $O(n^2)$
operations to determine all $n$ roots, and numerical evidence
strongly indicates that it is \emph{structured} backward stable;
it should be observed that at this time, our stability analysis
is incomplete. The square domain makes adaptive rootfinding possible, and is suitable for functions with nearby singularities outside of the square.
The performance of the scheme is illustrated 
via several numerical examples.

The approach of this paper can be seen as a generalization of the classical approach -- finding roots of polynomials on the real line by computing eigenvalues of \emph{companion matrices}, constructed from polynomial coefficients in the monomial basis. Since the monomial basis is ill-conditioned on the real line, it is often desirable to use the Chebyshev polynomial basis~\cite{trefethen2019approximation}. By analogy with the companion matrix, roots of polynomials can be obtained as eigenvalues of ``colleague matrices'', which are formed from the Chebyshev  expansion coefficients. Both the companion and colleague matrices have special structures that allow them to be represented by only a small number of vectors, rather than a dense matrix, throughout the QR iterations used to calculate all eigenvalues of the matrix.
Based on such representations, special-purpose QR algorithms that are proven \emph{structured} backward stable have been designed (see \cite{aurentz2018fast} and  \cite{aurentz2015fast}  for the companion matrix, and \cite{serkh2021} for the colleague matrix).

\hl{The success of the colleague matrix approaches hinges on two key facts. First, the existence of orthogonal polynomials on the real line, such as the Chebyshev polynomials, offers well-conditioned bases for accurately expanding reasonably smooth functions. Second, a three-term recurrence relation is automatically associated with such an orthogonal basis, so its colleague matrix is of the form of a Hermitian tridiagonal matrix plus a rank-one update. This algebraic structure is what enables the construction of the special QR algorithms in \cite{serkh2021}, achieving superior numerical stability and efficiency. As a result, the difficulties in extending such approaches to the complex plane become self-evident; in general domains in the complex plane, there are no orthogonal polynomials satisfying three-term recurrence relations. This can be understood by observing that the self-adjointness of multiplication by a real $x$ (with respect to standard inner products), which guarantees three term-recurrence relations for orthogonal polynomials on real intervals, no longer holds in the complex plane since the real $x$ is replaced by a complex $z$. Consequently, in the complex plane, colleague matrices in the classical sense do not exist, and compromises on the properties of polynomials bases have to be made. 

An obvious approach would be to replace the classical inner product with the ``complex inner product'', omitting the conjugation. A simple analysis shows that the resulting theory is highly unstable. For example, an attempt to construct ``complex orthogonal'' polynomials on any square centered at the origin fails immediately since the integral of $z^2$ over such a square is zero.
In this paper, we observe that replacing the ``complex inner product'' with a ``complex inner product'' with a random weight function greatly alleviates the problem. The resulting bases are \emph{not} orthogonal in the classical sense but are still reasonably well-conditioned.
The bases naturally lead to generalized colleague matrices very similar to the classical case. We then extend the QR algorithm of \cite{serkh2021} to such generalized colleague matrices.

The structure of this paper is as follows. Section~\ref{sec:prelim} contains the mathematical preliminaries, where a nonstandard unconjugated inner product is introduced. 
Section~\ref{sec:numerical} contains the numerical apparatus to be used in the remainder of the paper, including the construction of special polynomial bases.
Section~\ref{sec:QR} builds the complex orthogonal version of the QR algorithm in \cite{serkh2021}. Section~\ref{sec:algorithm} contains a description of the rootfinding algorithm for analytic functions in a square. Results of several numerical experiments are shown in Section~\ref{sec:result}. Finally, generalizations and conclusions can be found in Section~\ref{sec:conclude}. }

\section{Mathematical preliminaries \label{sec:prelim}}

\subsection{Notation}
We will denote the transpose of a complex matrix $A$ by $A^T$, and emphasize that $A^T$ is the \emph{transpose} of $A$, as opposed to the standard practice, where $A^*$ is obtained from $A$ by transposition \emph{and} complex conjugation. Similarly, we denote the \emph{transpose} of a complex vector $b$ by $b^T$, as opposed to the standard $b^*$ obtained by transposition \emph{and} complex conjugation. Furthermore, we denote standard $l^2$ norms of the complex matrix $A$ and the complex vector $b$ by $\norm{A}$ and $\norm{b}$ respectively.

\subsection{Maximum principle}

Suppose $f(z)$ is an analytic function in a compact domain $D\subset\C$. 
Then the maximum principle states the maximum modulus $\abs{f(z)}$ in $D$
is attained on the boundary $\partial D$.
The following is the famous maximum principle for complex analytic functions (see, for example, \cite{stein2010complex}).

\begin{theorem}
Given  a compact domain $D$ in $\mathbb{C}$ with boundary $\partial D$,  the modulus $|f(z)|$ of a non-constant analytic function $f$ in $D$ attains its maximum value on $\partial D$.	
\label{thm:max}
\end{theorem}

\subsection{Complex orthogonalization}
Given two complex vectors $u,v\in\C^{m}$,  we define a complex inner product $[u,v]_{w}$ of $u,v$ with respect to complex weights $w_1,  w_2,  \ldots, w_m \in \C$ by the formula
\begin{equation}
	\sb{u,v}_{w}=\sum_{j=1}^m w_j u_j v_j \,.
	\label{eq:com_prod}
\end{equation}
By analogy with the standard inner product, the complex one in \eqref{com_prod} admits notions of norms and orthogonality. Specifically, the complex norm of a vector $u\in\C^m$ with respect to $w$ is given by 
\begin{equation}
	\sb{u}_w = \sqrt{\sum_{i=1}^m  w_i u_i^2 }\,.
	\label{eq:cnorm}
\end{equation}
All results in this paper are independent of the branch chosen for the square root in \eqref{cnorm}.

We observe that the complex inner product behaves very differently from the standard inner product $\langle u,v\rangle=\sum_{j=1}^m \overline{u}_j v_j$,  where $\overline{u}_i$ is the complex conjugate of $u_i$. The elements in $w$ serving as the weights are complex instead of being positive real, and complex orthogonal vectors $u,v$ with $[u,v]_w=0$ are not necessarily linearly independent. Unlike the standard norm of $u$ that is always nonzero unless $u$ is a zero vector,  
the complex norm in \eqref{cnorm} can be zero even if $u$ is nonzero  (see \cite{craven1969complex} for a detailed discussion).

Despite this unpredictable behavior,  in many cases, complex orthonormal bases can be constructed via the following procedure.
Given a complex symmetric matrix $Z\in\C^{m\times m}$ (i.e.$Z=Z^T$), a complex vector $b\in\C^m$ and the initial conditions $q_{-1}=0, \beta_{0}=0$ and $q_0 = b/[b]_w$, the complex orthogonalization process
\begin{eqnarray}
	&&v = Zq_{j}, \label{eq:mult}\\
	&&\alpha_{j+1} = [q_j,v]_w  \label{eq:a}\\ 
	&&v = Zq_{j} -\beta_{j}q_{j-1} -\alpha_{j+1} q_j\,,\label{eq:zq}\\
	&&\beta_{j+1} = [v]_w\,, \label{eq:b}\\
	&&q_{j+1} = v/\beta_{j+1}\label{eq:normal}\,,
\end{eqnarray}
for $j=0,1,\ldots,n$ ($n\le m-1$)\,, produces vectors $q_0, q_1, \ldots, q_n$ together with  two complex coefficient vectors $\alpha = (\alpha_1,\alpha_2,\ldots,\alpha_n)^T$ and $\beta = (\beta_1,\beta_2,\ldots,\beta_n)^T$, provided the iteration does not break down (i.e.  none of the denominators $\beta_j$ in \eqref{normal} is zero). 
It is easily verified that these constructed vectors are complex orthonormal:
	\begin{equation}
		[q_i]_w = 1 \quad \mbox{for all $i$},  \quad [q_i,q_j]_w = 0\quad \mbox{for all $i\ne j$}\,. 
	\end{equation}
Furthermore, substituting $\beta_{j+1}$ and $q_{j+1}$ defined respectively in \eqref{b} and \eqref{normal} into \eqref{zq} shows that the vectors $q_j$ satisfy a three-term recurrence
	\begin{equation}
		Zq_j = \beta_{j}q_{j-1} + \alpha_{j+1} q_j + \beta_{j+1} q_{j+1},  
	\end{equation}
defined by coefficient vectors $\alpha$ and $\beta$. 
The following lemma summarizes the properties of the vectors generated  via  this procedure.

\begin{lemma}
	\label{lem:lanczos}
	Suppose $b\in \C^m$ is an arbitrary complex vector,  $w\in \C$ is an $m$-dimensional complex vector,  and $Z\in\C^{m\times m}$ with $Z=Z^T$ is an $m\times m$ complex symmetric matrix.  
	Provided the complex orthogonalization from \eqref{mult} to \eqref{normal} does not break down ,  there exists a sequence of $n+1$ ($n\le m-1$) complex orthonormal vectors $q_0,  q_1,  \ldots,  q_{n} \in \C^m$ with 
	\begin{equation}
		[q_i]_w = 1 \quad \mbox{for all $i$},  \quad [q_i,q_j]_w = 0\quad \mbox{for all $i\ne j$}\,,
	\end{equation}
	that satisfies a three-term recurrence relation,  defined by two complex coefficient vectors $\alpha,  \beta \in \C^n$,  
	\begin{equation}
		Zq_j = \beta_{j}q_{j-1} + \alpha_{j+1} q_j + \beta_{j+1} q_{j+1}, \quad \mbox{for  $0 \le j \le n-1$}\,,
		\label{eq:3-term}
	\end{equation}
	with the initial conditions
	\begin{equation}
		q_{-1}=0\,,\beta_0 = 0\quad \mbox{and}\quad q_0 = b/[b]_w\,.
		\label{eq:initial}
	\end{equation}
	
\end{lemma}


\subsection{Linear algebra}
The following lemma states that if the sum of a complex symmetric matrix $A$ and a rank-1 update $pq^T$ is lower Hessenberg (i.e. entries above the superdiagonal are zero),  then the matrix $A$ has a representation in terms of its diagonal,  superdiagonal and the vectors $p,  q$. It is a simple modification of the case of a Hermitian $A$ (see, for example, \cite{eidelman2008efficient}).
\begin{lemma}
	Suppose that $A\in \mathbb{C}^{n\times n}$ is complex symmetric (i.e. $A=A^T$),  and let $d$ and $\tau$ denote the diagonal and superdiagonal of $A$ respectively.  Suppose further that $p,  q\in  \mathbb{C}^{n}$ and $A+pq^T$ is lower Hessenberg.  Then 
	\begin{equation}
		a_{ij} = \begin{cases}
		  -p_i q_j, & \text{if } j>i+1\,,\\
  		  \tau_i & \text{if } j=i+1\,,\\
    		  d_i & \text{if } j=i\,,\\
    		  \tau_i & \text{if } j=i-1\,,\\
		  -q_i p_j, & \text{if } j<i-1\,,
		\end{cases}
		\label{eq:structA}
	\end{equation} 
where $a_{ij}$ is the $(i,j)$-th entry of $A$.
	\label{lem:tridiag}
\end{lemma}

\begin{proof}
	Since the matrix $A+pq^T$ is lower Hessenberg, elements above the superdiagonal are zero, i.e. $A_{ij} + p_i q_j = 0$ for all $i,j$ such that $j>i+1$. In other words, we have $A_{ij} = -p_i q_j$ for $j>i+1$. By the complex symmetry of $A$, we have $ A_{ji} = -p_i q_j $ for all $i,j$ such that $j>i+1$, or $ A_{ij} = -q_i p_j $ for $i>j+1$ by exchanging $i$ and $j$. We denote the $i$th diagonal and superdiagonal (which is identical the subdiagonal) of $A$ by $d_i$ and $\tau_i$, and obtain the form of $A$ in (\ref{eq:structA}).
\end{proof}

The following lemma states if a sum of a matrix $B$ and a rank-1 update $pq^T$ is lower triangular, the upper Hessenberg part of $B$ (i.e. entries above the subdiagonal) has a representation entirely in terms of its diagonal,  superdiagonal elements and the vectors $p,  q$. It is straightforward modification of an observation in \cite{serkh2021}.
\begin{lemma}
	Suppose that $B\in\mathbb{C}^{n\times n}$ and let $d\in\C^n$, $\gamma\in\C^{n-1}$ denote the diagonal and subdiagonal of $B$ respectively.  Suppose further that $p,q\in\mathbb{C}^{n}$ and the matrix $B+pq^T$ is lower triangular.  Then
		\begin{equation}
		b_{ij} = \begin{cases}
		  -p_i q_j, & \text{if } j>i\,,\\
    		  d_i & \text{if } j=i\,,\\
    		  \gamma_i & \text{if } j=i-1\,.\\
		\end{cases}
		\label{eq:triang}
	\end{equation} 
		\label{lem:triang}
\end{lemma}

\begin{proof}
	Since the matrix $B+pq^T$ is lower triangular, elements above the diagonal are zero, i.e. $B_{ij} + p_i q_j = 0$ for all $i,j$ such that $j>i$. In other words, $B_{ij} = -p_i q_j $ whenever $j>i$. We denote the $i$th diagonal and subdiagonal by vectors $d_i$ and $\gamma_i$, and obtain the representation of the upper Hessenberg part of $B$ in (\ref{eq:triang}).
\end{proof}

The following lemma states the form of a $2\times2$ complex orthogonal matrix $Q$ that eliminates the first component of an arbitrary complex vector $x\in\C^2$. We observe that elements of the matrix $Q$ can be arbitrarily large, unlike the case of classical SU(2) rotations.
\begin{lemma}
	Given $x=\bb{x_1,x_2}^T\in\mathbb{C}^2$ such that $x_1^2+x_2^2\ne 0$, suppose we define $c=\frac{x_2}{\sqrt{x_1^2+x_2^2}}$ and $s=\frac{x_1}{\sqrt{x_1^2+x_2^2}}$,  or  $c=1$ and $s=0$ if $\norm{x}=0$. 
	Then the $2\times 2$ matrix $Q$ defined by	
	\begin{eqnarray}
		&&Q_{11}=c \quad Q_{12}=-s\\
		&&Q_{21}=s\quad Q_{22}=c
	\end{eqnarray}
	eliminates the first component of $x$:
	\begin{equation}
		Q x = \bb{\mat{c & -s \\ s & c}}\bb{\mat{x_1\\x_2}} = \bb{\mat{0\\ \sqrt{x_1^2+x^2_2}}}\,.
	\end{equation}	
	Furthermore,  $Q$ is a complex orthogonal matrix (i.e. $Q^T=Q^{-1}$).
	\label{lem:rot}
\end{lemma}


%
%

\section{Numerical apparatus\label{sec:numerical}}
In this section, we develop the numerical apparatus used in the remainder of the paper. We describe a procedure for constructing polynomial bases that satisfy three-term recurrence relations in complex domains. Once a polynomial is expressed in one of the constructed bases, we use the expansion coefficients to form a matrix $C$, whose eigenvalues are the roots of the polynomial.
Section~\ref{sec:three} contains the construction of such bases. 
Section~\ref{sec:colleague} contains the expressions of the matrix $C$. 
We refer to matrices of the form $C$ as the ``generalized colleague matrices'', 
since they have structures similar to those of the classical ``colleague matrices''.

\subsection{Generalized colleague matrices \label{sec:colleague}}


Suppose $\alpha = (\alpha_1, \alpha_2, \ldots, \alpha_n)^T, \beta = (\beta_1, \beta_2, \ldots, \beta_n)^T \in \C^n$ are two complex vectors, and  $P_0, P_1, P_2, \ldots,P_n$ are polynomials such that the order of $P_j$ is $j$. Furthermore,  suppose those polynomials satisfy a three-term recurrence relation
\begin{equation}
	z P_j(z)=\beta_{j} P_{j-1}(z) + \alpha_{j+1}P_{j}(z) + \beta_{j+1} P_{j+1}(z)  \quad \mbox{for all} \quad 0\le j \le n-1\,,
	\label{eq:three_thm}
\end{equation}
with \hl{the definition $P_{-1}=0, \beta_0=0$ for convenience}. Given a polynomial $p$ of order $n$ defined by the formula
\begin{equation}
	p(z)=\sum_{j=0}^{n}c_j P_j(z)\,,
	\label{eq:polyexp}
\end{equation}
the following theorem states the $n$ roots of \eqref{polyexp} are eigenvalues of an $n\times n$ ``generalized colleague matrix''. The proof is virtually identical to that of classical colleague matrices, which can be found in \cite{trefethen2019approximation}, Theorem~18.1.
\begin{theorem}
	The roots of the polynomial in \eqref{polyexp} are the eigenvalues of the matrix $C$  defined by the formula
	\begin{equation}
		C=A +   e_{n} q^T \,,
		\label{eq:colleague}
	\end{equation}
where $A$ is a complex symmetric tridiagonal matrix 
	\begin{equation}
		A=\bb{\mat{\alpha_1 & \beta_1 & & && \\
				\beta_1 & \alpha_2 &\beta_2 & & &\\
				&\beta_2&\alpha_3&\beta_3 & & \\
				& & \ddots &\ddots& \ddots & & \\ 	
				& & & \beta_{n-2} & \alpha_{n-1} &\beta_{n-1}  \\ 
				& & & & \beta_{n-1} &\alpha	_{n}}}\,,
				\label{eq:tridiag}
	\end{equation}
	and $e_n, q \in \C^n$ are defined, respectively, by the formulae
	\begin{equation}
		e_n=\bb{\mat{0&\cdots &0 & 1}}^T \,,
	\end{equation}
	and
	\begin{equation}
		q = -\beta_{n}\bb{\mat{\frac{c_0}{c_n}& \frac{c_1}{c_n} & \frac{c_2}{c_n}&\cdots & \frac{c_{n-1}}{c_n}}}^T\,.
		\label{eq:qvec}
	\end{equation}
	\label{thm:colleague}
\end{theorem}
We will refer to the matrix $C$ in (\ref{eq:colleague}) as a \emph{generalized colleague matrix}. Obviously, the matrix $C$  is lower Hessenberg.  It should be observed that the matrix $A$ in Theorem~\ref{thm:colleague} consists of complex entries, so it is complex symmetric, as opposed to Hermitian in the classical colleague matrix case. Thus, the matrix $C$ has the structure of a complex symmetric tridiagonal matrix plus a rank-1 update.  

\begin{remark}
\label{rmk:large}
	Similar to the case of classical colleague matrices, all non-zero elements in the matrix $A$  in (\ref{eq:colleague}) are of order 1. \hl{(The specific values depend on the choice of the polynomial basis.)} However, the matrix $e_n q^T$ in (\ref{eq:colleague}) may contain elements orders of magnitude larger. Indeed, when the polynomial $p$ in (\ref{eq:polyexp}) approximates an analytic function $f$ to high accuracy, the leading coefficient $c_n$ is likely to be small. In extreme cases, some of the elements of the vector $q$ in (\ref{eq:qvec}) will be of order $1/u$, with $u$ the machine epsilon. Due to this imbalance in the sizes of the two parts in the matrix $C$, the usual backward stability of QR algorithms is insufficient to guarantee the precision of eigenvalues computed. As a result, special-purpose QR algorithms are required. We refer readers to \cite{serkh2021} for a more detailed discussion and the structured backward stable QR for classical colleague matrices.
	
\end{remark}

%

\newpage
\subsection{Complex orthogonalization for three-term recurrences\label{sec:three}}

In this section, we describe an empirical observation at the core of our rootfinding scheme, enabling the construction of a class of bases in simply connected compact domain in the complex plane. Not only are these bases reasonably well-conditioned, they obey three-term recurrence relations similar to those of classical orthogonal polynomials. Section~\ref{sec:poly_three} contains a procedure for constructing polynomials satisfying three-term recurrences in the complex plane, based on the complex orthogonalization of vectors in Lemma~\ref{lem:lanczos}. It also contains the observation that the constructed bases are reasonably well-conditioned for most practical purposes if random weights are used in defining the complex inner product. This observation is illustrated numerically in Section~\ref{sec:data}. 



\subsubsection{Polynomials defined by three-term recurrences in complex plane \label{sec:poly_three}}
Given $m$ points $z_1,z_2,\ldots,z_m$ in the complex plane, and $m$ complex numbers $w_1, w_2,\ldots,w_m$,
we form a diagonal matrix $Z\in \C^{m\times m}$ by the formula
\begin{equation}
		Z=\mathrm{diag}(z_1, z_2, \ldots, z_m)\,,
		\label{eq:Z}
\end{equation}
(obviously, $Z$ is complex symmetric, i.e. $Z=Z^T$),  and a complex inner product in $\C^m$ by choosing $\cb{w_j}_{j=1}^m$ as weights in (\ref{eq:com_prod}). Suppose further that we choose the initial vector $b=\bb{1,1,\ldots,1}^T\in\C^m$. The matrix $Z$ and the vector $b$ define a complex orthogonalization process described in Lemma~\ref{lem:lanczos}.
Provided the complex orthogonalization process does not break down, by Lemma~\ref{lem:lanczos} it generates
complex orthonormal vectors $q_0,q_1,\ldots,q_n\in\C^m$, along with two vectors $\alpha,\beta\in\C^n$ (See (\ref{eq:3-term})); vectors  $q_0,q_1,\ldots,q_n$ satisfy the three-term recurrence relation 
\begin{equation}
	z_i (q_j)_i =\beta_{j}  (q_{j-1})_i + \alpha_{j+1} (q_j)_i + \beta_{j+1}  (q_{j+1})_i\,,
	\label{eq:threeq}
   \end{equation}
for $j=0,1,\ldots,n-1$ and for each component $i = 1,2,\ldots,m$\,,
with initial conditions of the form (\ref{eq:initial}). In particular, all components of $q_0$ are identical. The vector $q_{j+1}$ is constructed as a linear combination of vectors $Zq_j, q_j$ and $q_{j-1}$ by (\ref{eq:zq}). 
Moreover, the vector $Zq_j$ consists of components in $q_j$ multiplied by $z_1, z_2, \ldots, z_m$ respectively (See (\ref{eq:Z}) above). By induction, components of the vector $q_j$ are values of a polynomial of order $j$ at points $\cb{z_i}_{i=1}^m$ for $j=0,1,\ldots,n$.  We thus define polynomials $P_0, P_1, P_2,\ldots,P_{n}$, such that $P_j$ is of order $j$ with $P_j(z_i) = (q_j)_i$ for $i=1,2,\ldots,m$. As a result, the relation (\ref{eq:threeq}) becomes a three-term recurrence for polynomials $\cb{P}_{i=0}^n$ restricted to points $\cb{z_i}_{i=1}^m$. Since the polynomials are defined everywhere, we extend the polynomials and the three-term recurrence to the entire the complex plane. 
We thus obtain polynomials $P_0, P_1, P_2,\ldots,P_{n}$ that satisfy a three-term recurrence relation in (\ref{eq:three_thm}). We summarize this construction in the following lemma.

\begin{lemma}
	\label{lem:ttr_poly}
	Let $z_1, z_2, \ldots, z_m$ and $w_1,w_2,\ldots,w_m$ be two sets of complex numbers. 
We define a diagonal matrix $Z=\mathrm{diag}(z_1, z_2, \ldots, z_m)$, and a complex inner product of the form (\ref{eq:com_prod}) by choosing $w_1,w_2,\ldots,w_m$ as weights. If the complex orthogonalization procedure defined above in Lemma~\ref{lem:lanczos} does not break down, there exist polynomials $P_0, P_1, P_2,\ldots, P_n$ $(n\le m-1)$ such that $P_j$ is of order $j$,  with two vectors $\alpha, \beta\in \C^n$ (See (\ref{eq:3-term})). In particular, the polynomials satisfy a three-term recurrence relation of the form (\ref{eq:three_thm}) in the entire complex plane.
\end{lemma}

\begin{remark}
We observe that the matrix $Z$ in (\ref{eq:Z}) is complex symmetric, i.e. $Z=Z^T$, as opposed to being Hermitian. Thus, if we replace the complex inner product used in the construction in Lemma~\ref{lem:ttr_poly} with the standard inner product (with complex conjugation), it is necessary to orthogonalize $Zq_j$ in (\ref{eq:zq}) against not only $q_{j}$ and $q_{j-1}$ but also all remaining preceding vectors $q_{j-2}, \ldots, q_0$, since the complex symmetry of $Z$ is unable to enforce orthogonality automatically, destroying the three-term recurrence. As a result, the choice of complex inner products is necessary for constructing three-term recurrences. 
\end{remark}

Although the procedure in Lemma~\ref{lem:ttr_poly} provides a formal way of constructing polynomials satisfying three-term recurrences in $\C$, it will not in general produce a practical basis. The sizes of the polynomials tend to grow rapidly as the order increases, even at the points where the polynomials are constructed. Thus, viewed as a basis, they tend to be extremely ill-conditioned. In particular, when the points $\cb{z_i}_{i=1}^m$ are Gaussian nodes of each side of a square (with weights $\cb{w_i}_{i=1}^m$ the corresponding Gaussian weights), the procedure in Lemma~\ref{lem:ttr_poly} breaks down due to the complex norm $[q_1]_w$ of the first vector produced being zero. This problem of constructed bases being ill-conditioned is partially solved by the following observations.

We observe that, if the weights $\cb{w_i}_{i=1}^m$ in defining the complex inner product are chosen to be \emph{random} complex numbers, the growth of the size of polynomials at points of construction $\cb{z_i}_{i=1}^m$ is suppressed -- the polynomials form reasonably well-conditioned bases and the condition number grows almost linearly with the order.
Moreover, this phenomenon is relatively insensitive to locations of points $\cb{z_i}_{i=1}^m$ and distributions from which numbers $\cb{w_i}_{i=1}^m$ are drawn. For simplicity, we choose $\cb{w_i}_{i=1}^m$ to be real random numbers uniformly drawn from $[0,1]$ for all numerical examples in the remaining of this paper. 
Thus, this observation combined with Lemma~\ref{lem:ttr_poly} enables the construction of practical bases satisfying three-term recurrences. Furthermore, when such a basis is constructed based on points $\cb{z_i}_{i=1}^m$ chosen on the boundary $\partial D$ of a compact domain $D\subset \C$, we observe that both the basis and the three-term recurrence extends \emph{stably} to the {entire} domain $D$. It is not completely understood why the choice of random weights leads to well-conditioned polynomial bases and why their extension over a compact domain is stable. Both questions are under vigorous investigation.


\begin{remark}
	For rootfinding purposes in a compact domain $D$, since the three-term recurrences extend stably over the entire $D$ automatically, we only need values of the basis polynomials at points on its boundary $\partial D$. In practice, those points are chosen to be identical to those where the polynomials are constructed (See Lemma~\ref{lem:ttr_poly} above), so their values at points of interests are usually readily available. Their values elsewhere, if needed, can be obtained efficiently by evaluating their three-term recurrences.
\end{remark}



\subsubsection{Condition numbers of $\cb{P_j} $ \label{sec:data}}
In this section, we evaluate numerically condition numbers of polynomial bases constructed via Lemma~\ref{lem:ttr_poly}. Following the observations discussed above, random numbers are used as weights in defining the complex inner product for the construction.
 We construct such polynomials at $m$ nodes $\cb{z_j}_{j=1}^m$ on the boundary of a square, a triangle, a snake-like domain and a circle in the complex plane. Clearly, the domains in the first three examples are polygons. For choosing nodes $\cb{z_j}_{j=1}^m$ on the boundaries of these polygons, we parameterize each edge by the interval $[-1,1]$, by which each edge is discretized with equispaced or Gaussian nodes. We choose the number of discretization points on each edge to be the same (so the total node number $m$ is different depending on the number of edges).  In Figure~\ref{fig:shapes}, discretizations with Gaussian nodes are illustrated for the first three examples, and that with equispaced nodes is shown in the last one.
Plots of polynomials $P_2, P_5$ and $P_{20}$ are shown in Figure~\ref{fig:basis}, and they are constructed with Gaussian nodes on the square shown in Figure~\ref{fig:square} (for a particular realization of random weights). Roots of polynomial $P_{300}$ are also shown in Figure~\ref{fig:basis} for all examples constructed with Gaussian nodes, where most roots of $P_{300}$ are very close to the exterior boundary on which the polynomials are constructed.
 
In the first three examples, to demonstrate the method's insensitivity to the choice of $\cb{z_j}_{j=1}^m$, we construct polynomials at points $\cb{z_j}_{j=1}^m$ formed by both equispaced and Gaussian nodes on each edge.
In the last (circle)  example, the boundary is only discretized with $m$ equispaced nodes. We measure the condition numbers of the polynomial bases by that of the $m\times (n+1)$ basis matrix, denoted by $G$, formed by values of polynomials $\cb{P_j}_{j=0}^n$ at Gaussian nodes $\cb{\tilde{z}_j}_{j=1}^m$ on each edge (not necessarily the same as $\cb{{z}_j}_{j=1}^m$), scaled by the square root of the corresponding Gaussian weights $\cb{\tilde{w}_j}_{j=1}^m$ -- except for the last example, where the nodes $\cb{\tilde{z}_j}_{j=1}^m$ are equispaced, and all weights $\cb{\tilde{w}_j}_{j=1}^m$ are $2\pi/m$. More explicitly, the $(i,j)$ entry of the matrix $G$ is given by the formula
\begin{equation}
	G_{ij} = \sqrt{\tilde{w}_i}P_j(\tilde{z}_i)\,.
	\label{eq:GG}
\end{equation}
It should be observed that the weights $\cb{\tilde{w}_j}_{j=1}^m$ in $G$ are different from those for constructing the polynomials -- the former are good quadrature weights while the latter are random ones.
Condition numbers of the first three examples shown Figure~\ref{fig:sq_plot}, labeled ``Gaussian 1'' and ``Gaussian 2'', are those of the bases constructed at Gaussian nodes on each edge. Thus the construction points coincide with those for condition number evaluations, and the matrix $G$ is readily available. Condition numbers labeled ``Equispaced'' in the first three examples are those of the bases constructed at equispaced nodes on each edge. In these cases, the corresponding three-term recurrences are used to obtain the values of the bases at Gaussian nodes, from which $G$ is obtained. In the last example, the points for basis construction also coincide with those for condition number evaluations, so $G$ is available directly.

Since the weights used in defining complex inner products during construction are random, the condition numbers are averaged over 10 realizations for each basis of order $n$.
It can be seen in Figure~\ref{fig:sq_plot} that the condition number grows almost linearly with the order $n$, an indication of the stability of the constructed bases and the method's domain independence. 
\hl{Although constructing high order polynomial bases ($n\ge 300$) can be easily done, we do not use bases of order larger than $n=100$ in all numerical experiments involving rootfinding in this paper, so the condition numbers of bases for rootfinding are no larger than 1000.}

\begin{figure}[htb]
	\centering

	\begin{subfigure}[t]{0.45\textwidth}
		\centering
		\includegraphics[width=\textwidth]{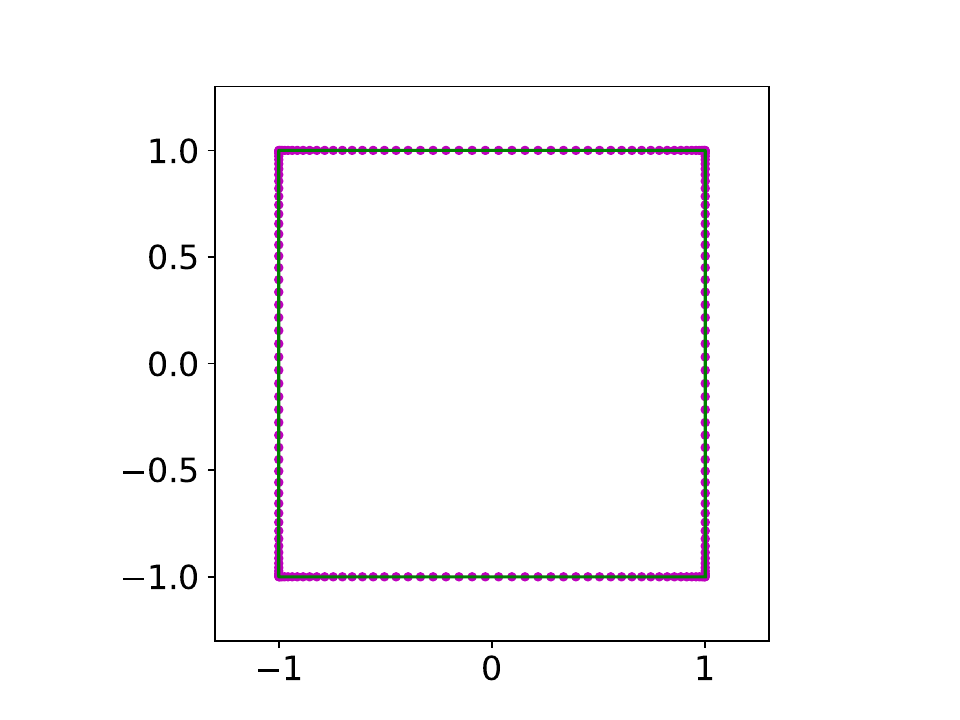}
		\caption{}
		\label{fig:square}
	\end{subfigure}	
	~
		\begin{subfigure}[t]{0.45\textwidth}
		\centering
		\includegraphics[width=\textwidth]{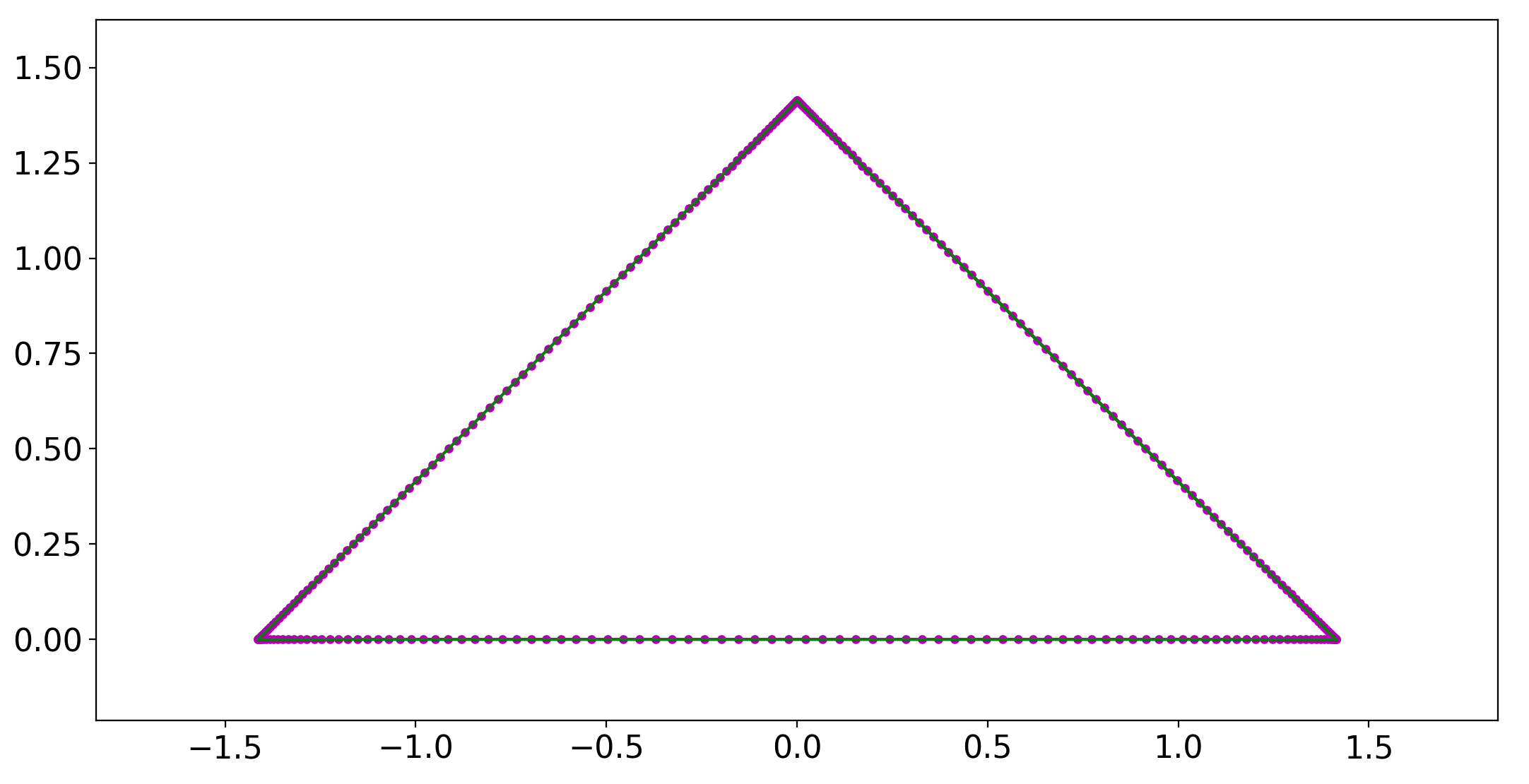}
		\caption{}
		\label{fig:triangle}
	\end{subfigure}	

	\begin{subfigure}[t]{0.5\textwidth}
		\centering
		\includegraphics[width=\textwidth]{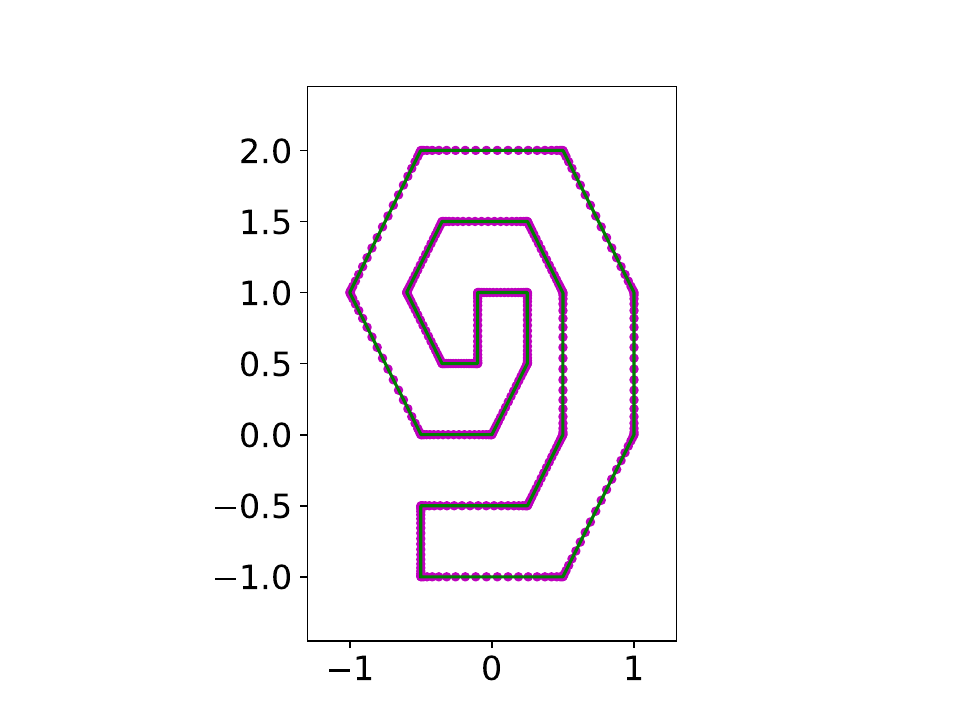}
		\caption{}
		\label{fig:snake}
	\end{subfigure}
	~
		\begin{subfigure}[t]{0.45\textwidth}
		\centering
		\includegraphics[width=\textwidth]{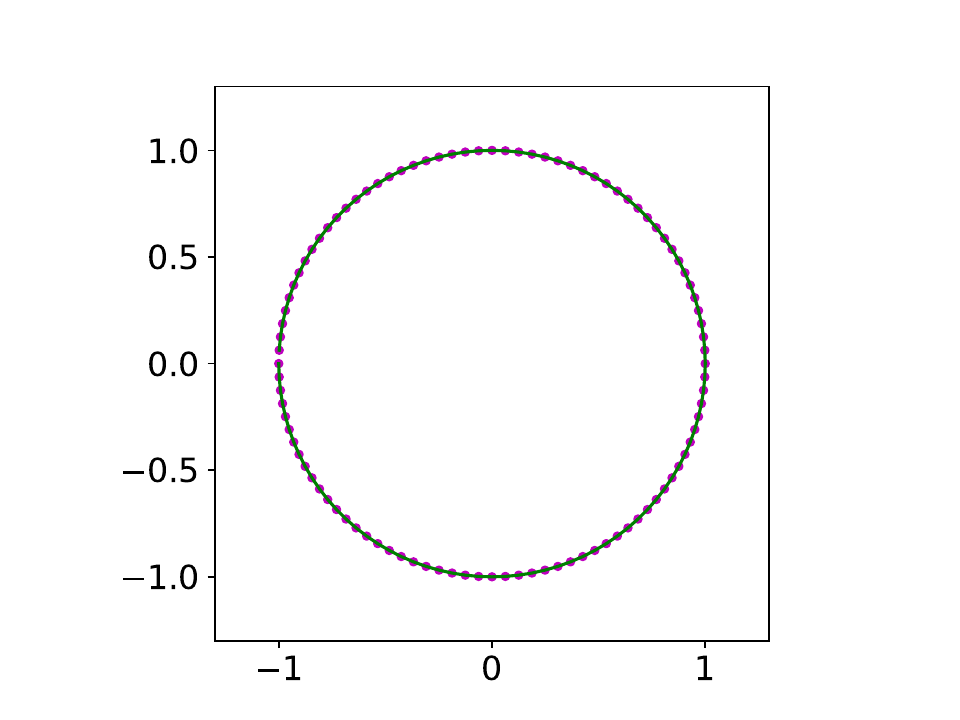}
		\caption{}
		\label{fig:circle}
	\end{subfigure}

	\caption{We construct the polynomial bases on a square, a triangle, a snake-like domain and a circle.
	\label{fig:shapes}}
\end{figure}

\clearpage

\begin{figure}[H]
	\centering

	\begin{subfigure}[t]{0.48\textwidth}
		\centering
		\includegraphics[width=\textwidth]{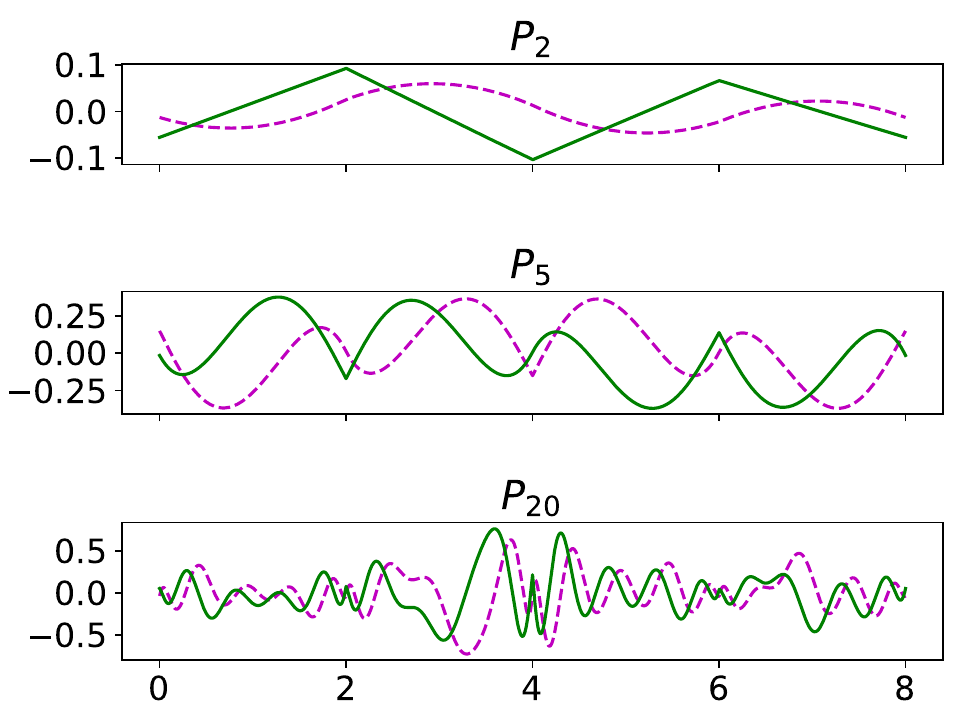}
		\caption{}
		\label{fig:basisplot}
	\end{subfigure}

	\begin{subfigure}[t]{0.48\textwidth}
		\centering
		\includegraphics[width=\textwidth]{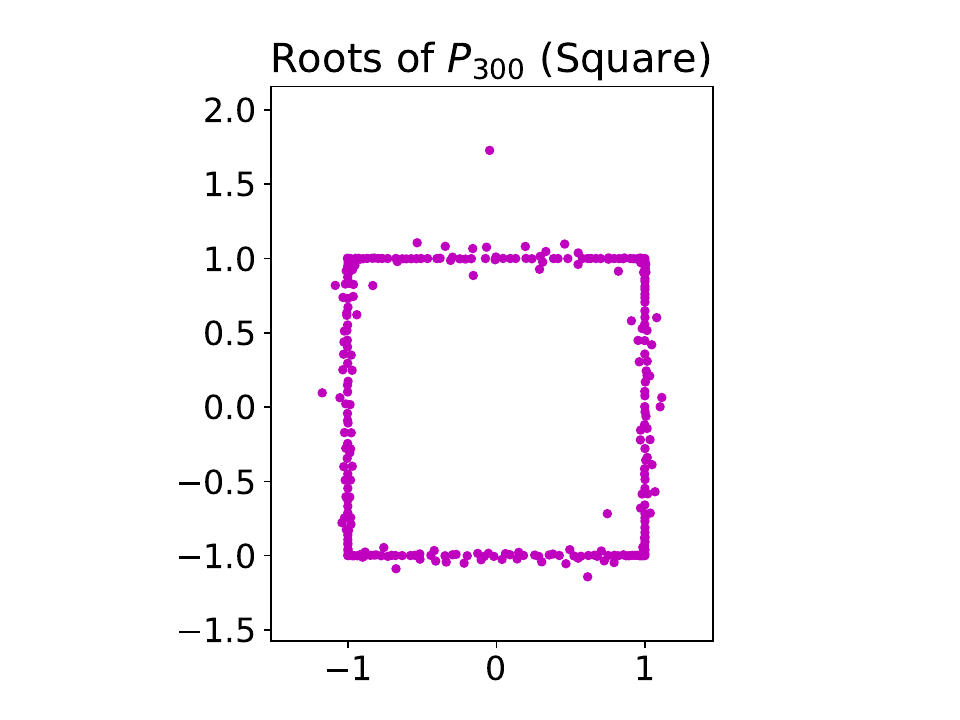}
		\caption{}
		\label{fig:basisroots}
	\end{subfigure}
	\begin{subfigure}[t]{0.48\textwidth}
		\centering
		\includegraphics[width=\textwidth]{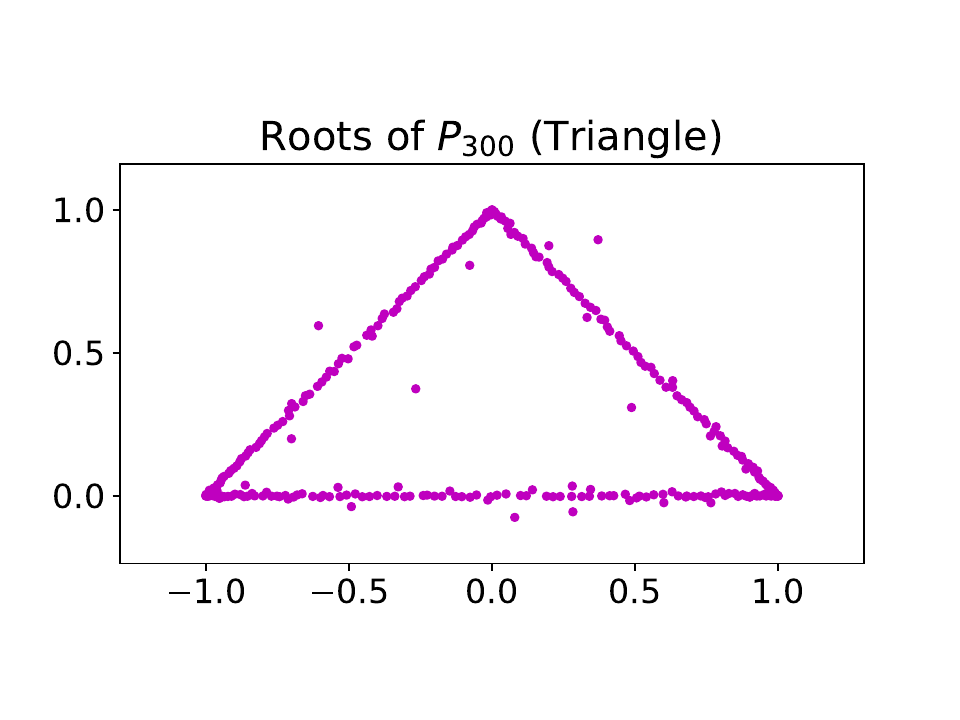}
		\caption{}
		\label{fig:basisroots_tri}
	\end{subfigure}
	\begin{subfigure}[t]{0.48\textwidth}
		\centering
		\includegraphics[width=\textwidth]{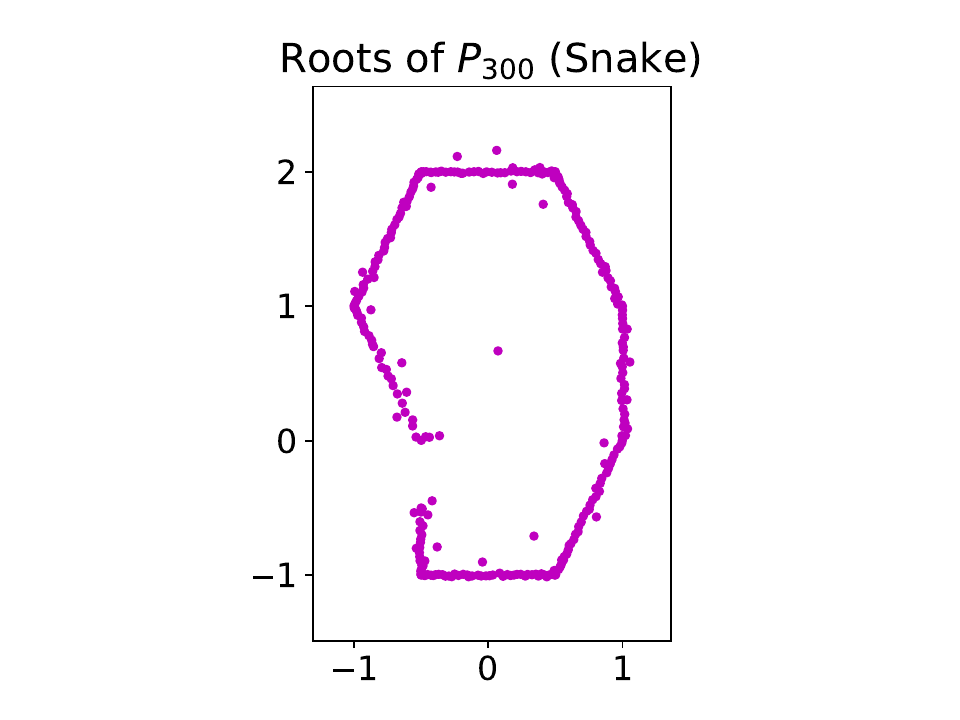}
		\caption{}
		\label{fig:basisroots_snake}
	\end{subfigure}
	\begin{subfigure}[t]{0.48\textwidth}
		\centering
		\includegraphics[width=\textwidth]{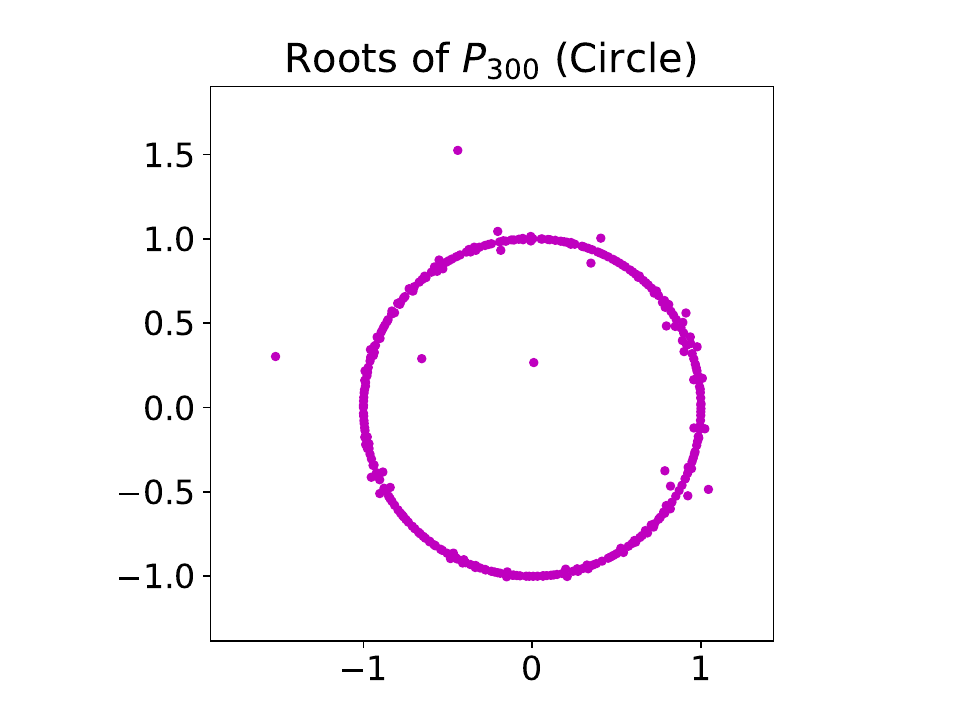}
		\caption{}
		\label{fig:basisroots_circle}
	\end{subfigure}	
	
	\caption{The graphs of the polynomials $P_2, P_5$ and $P_{20}$ are shown in (a), constructed on the boundary of a square domain shown in Figure~\ref{fig:square}. The four sides of the square are mapped to the interval $[0,8]$, on which the polynomials are plotted.
	Solid lines are the real part and dotted lines the imaginary part.
Roots of $P_{300}$ for all examples are shown. All roots tend to be near the exterior boundary of domains.
	\label{fig:basis}}
\end{figure}

\begin{figure}[h]
	\centering

	\begin{subfigure}[t]{0.4\textwidth}
		\centering
		\includegraphics[width=\textwidth]{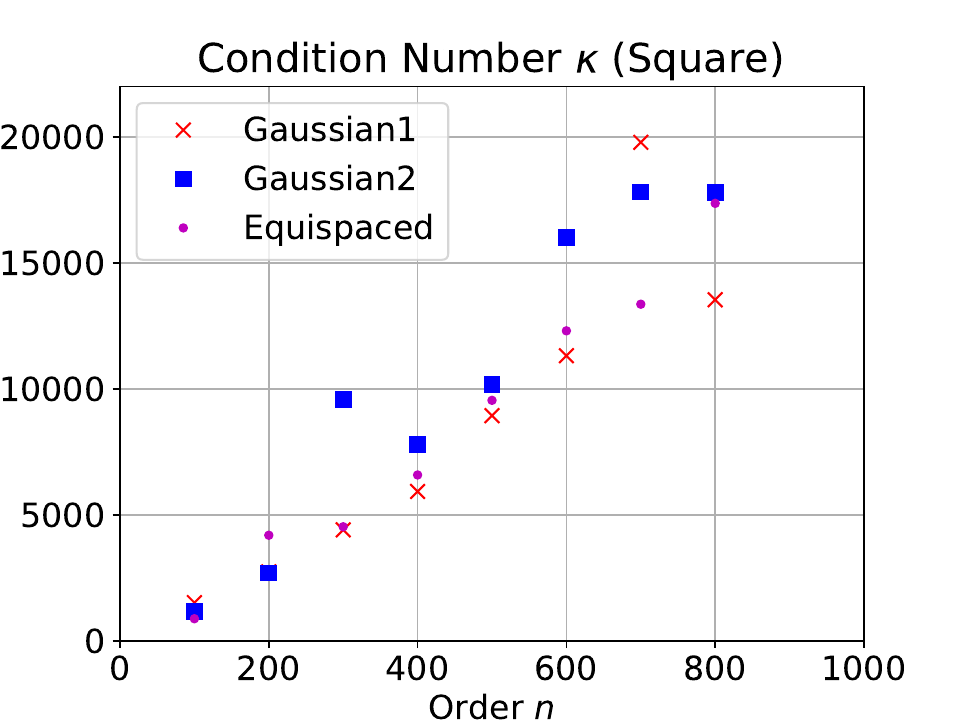}
		\caption{}
		\label{fig:sq_cond}
	\end{subfigure}
	~
	\begin{subfigure}[t]{0.4\textwidth}
		\centering
		\includegraphics[width=\textwidth]{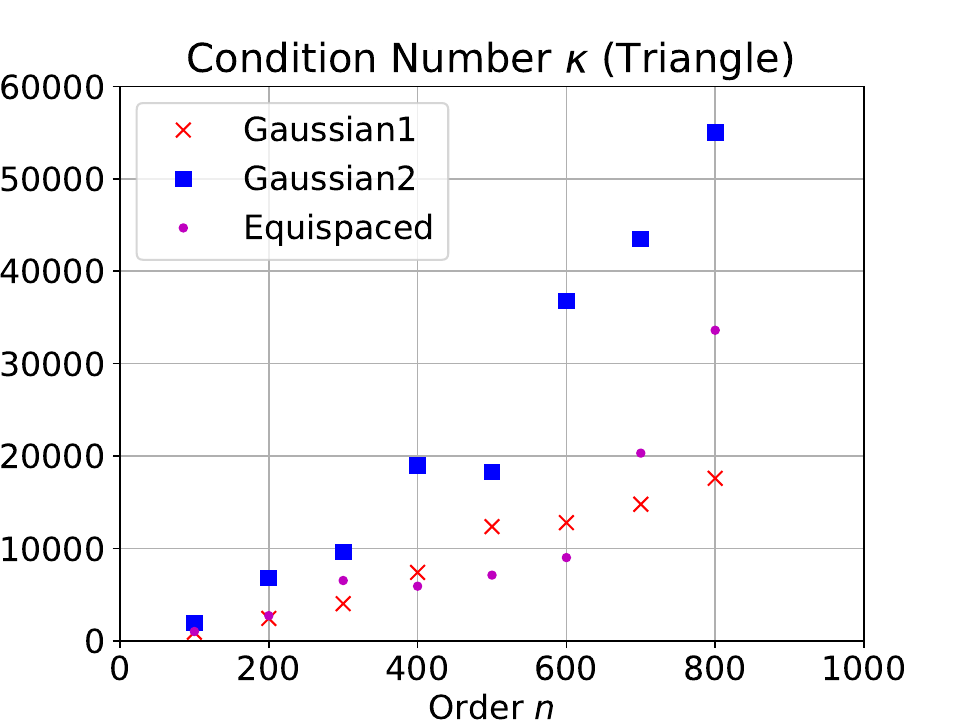}
		\caption{}
		\label{fig:tri_cond}
	\end{subfigure}

	\begin{subfigure}[t]{0.4\textwidth}
		\centering
		\includegraphics[width=\textwidth]{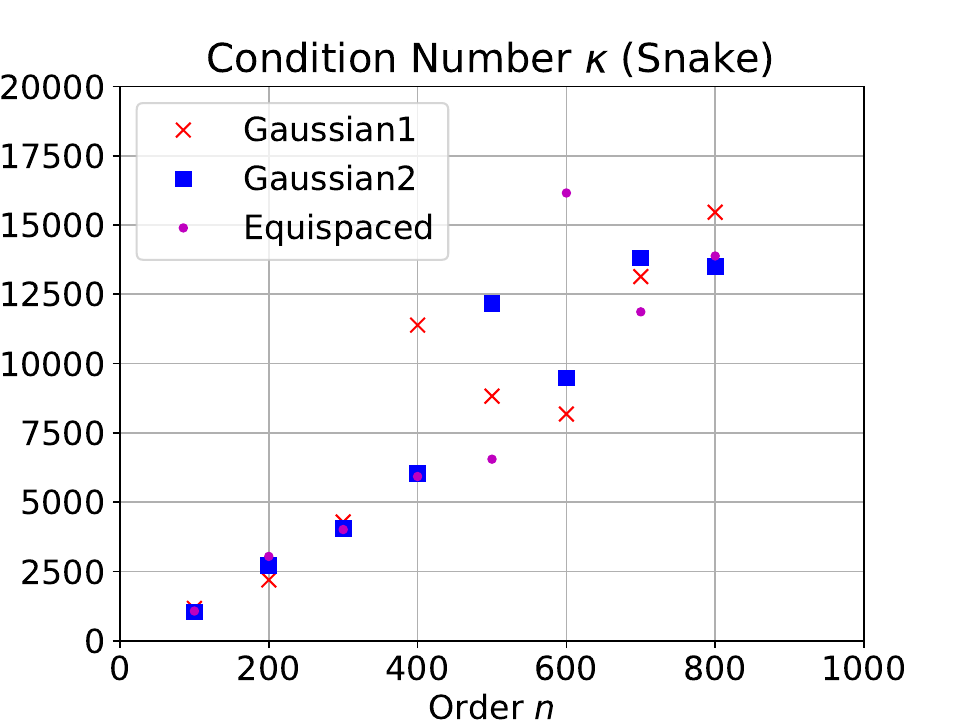}
		\caption{}
		\label{fig:snake_cond}
	\end{subfigure}
	~
	\begin{subfigure}[t]{0.4\textwidth}
		\centering
		\includegraphics[width=\textwidth]{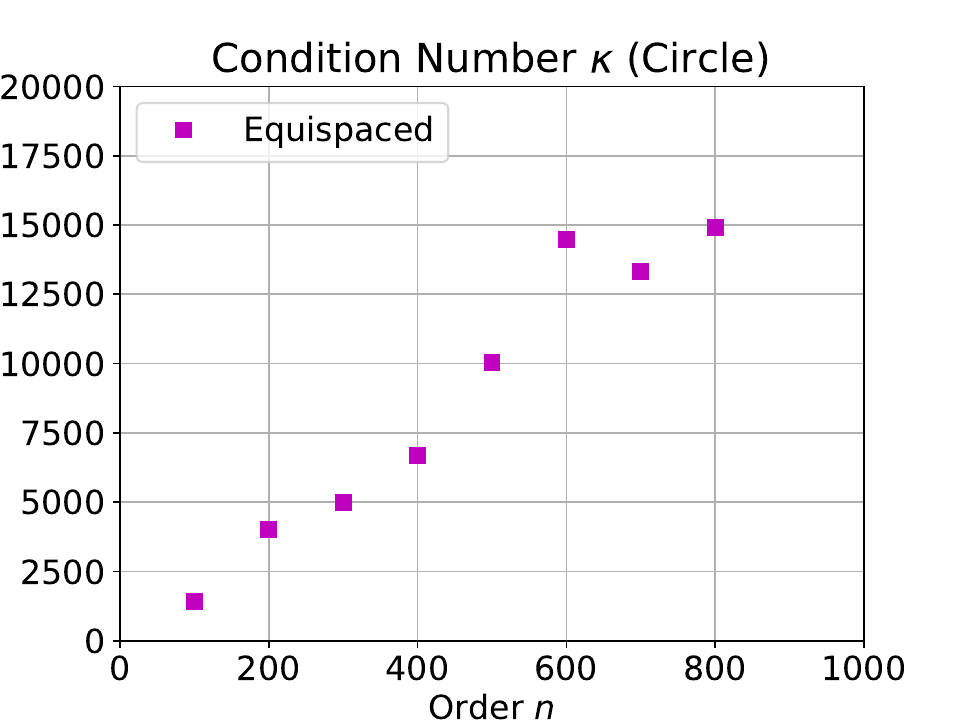}
		\caption{}
		\label{fig:c_cond}
	\end{subfigure}

	\caption{The condition numbers of polynomial bases constructed on the boundary of various domains in Figure~\ref{fig:shapes} are shown. Data points labeled by ``Gaussian 1''  for order $n$ represent condition numbers of $\cb{P_j}_{j=0}^n$ constructed with $n$ Gaussian nodes on each side; ``Gaussian 2'' represents those with $n/2+10$ Gaussian nodes. In both cases, the points where the polynomials are constructed coincide with those where the conditioned numbers of the polynomial bases are evaluated.
	 ``Equispaced'' represents those constructed with $n$ equispaced nodes on each side, except for the triangle example where $4n$ equispaced nodes are used.
	  In the last example, in total $2n$ equispaced nodes are used.
	\label{fig:sq_plot}}
\end{figure}

\newpage

\section{Complex orthogonal QR algorithm \label{sec:QR}}

This section contains the complex orthogonal QR algorithm used in this paper for computing the eigenvalues of the generalized colleague matrix $C$ in the form of (\ref{eq:colleague}). Section~\ref{sec:overview} contains an overview of the algorithm, and  Section~\ref{sec:qr:super} and \ref{sec:qr:sub} contain the detailed description.

The complex orthogonal QR algorithm is a slight modification of Algorithm 4 in \cite{serkh2021}. Similar to the algorithm in \cite{serkh2021}, the QR algorithm here takes $O(n^2)$ operations, and numerical results suggest the algorithm is structured backward stable.
The major difference is that $A$ in \eqref{tridiag} here is complex symmetric, whereas $A$ in \cite{serkh2021} is Hermitian.   In order to maintain the complex symmetry of $A$,  we replace SU(2) rotations used in the Hermitian case
with complex orthogonal transforms in Lemma~\ref{lem:rot}. (A discussion of similarity transforms via complex orthogonal matrices can be found in \cite{gantmakher2000theory}.) 
Except for this change,  the algorithm is very similar to its Hermitian counterpart. Due to the use of complex orthogonal matrices, the proof of backward stability in \cite{serkh2021} cannot be extended to this paper directly. The numerical stability of our QR algorithm is still under investigation. Thus here we describe the algorithm in exact arithmetic, except for the section describing an essential correction to stabilize large rounding error when the rank-1 update has large norm. 

\subsection{Overview of the complex QR \label{sec:overview}}
The complex orthogonal QR algorithm we use is designed for a class of lower Hessenberg matrices $\mathcal{F}\subset \C^{n\times n}$ of the form
\begin{equation}
	A+ p q^T\,,
	\label{eq:A}
\end{equation}
where $A\in \C^{n\times n }$ is complex symmetric and $p,  q\in \C^{n\times n}$.  Due to Lemma~\ref{lem:tridiag}, the matrix $A$ is determined entirely by:
\begin{enumerate}
	\item The diagonal entries $d_i = a_{i,i}$,  for $i=1,2,\ldots,n$,
	\item The superdiagonal entries $\beta_{i}=a_{i,  i+1}$ for $i=1,2,\ldots,n-1$,
	\item The vectors $p$, $q$\,.
\end{enumerate} 
Following the terminology in \cite{serkh2021},  we refer to the four vectors $d,\beta,p$ and $q$ as the \emph{basic elements} or \emph{generators} of $A$.  Below, we introduce a complex symmetric QR algorithm applied to a lower Hessenberg matrix of the form $C=A+pq^T$, defined by its generators $d,\beta,p$ and $q$. In a slight abuse of terminology, we call complex orthogonal transforms ``rotations'' as in \cite{serkh2021}. However, these complex orthogonal transforms are not rotations in general.

First, we eliminate the superdiagonal of the lower Hessenberg $C$. Let the matrix $U_n\in 
\C^{n\times n}$ be the complex orthogonal matrix (i.e. $U_n^T=U_n^{-1}$) that rotates the $(n-1,n)$-plane so that the superdiagonal in the  $(n-1,n)$ entry is eliminated:
\begin{equation}
	(U_n C)_{n-1,n} =0\,.
\end{equation}
We obtain the matrix $U_n$ from an $n\times n$ identity matrix, replacing its $(n-1,n)$-plane block with the $2\times 2$ complex orthogonal $Q_n$ computed via Lemma~\ref{lem:rot} for eliminating $C_{n-1,n}$ in the vector $(C_{n-1,n},C_{n,n})^T$.
Similarly, by the replacing $(n-2,n-1)$-plane block of an $n\times n$ identity matrix by the corresponding $Q_{n-1}$,
we form the complex orthogonal matrix $U_{n-1}\in\C^{n\times n}$ that rotates the $(n-2,n-1)$-plane to eliminate the superdiagonal in the $(n-2,  n-1)$ entry:
\begin{equation}
	(U_{n-1}U_{n}C)_{n-2,n-1} = 0.
\end{equation}
 We can repeat this process and use complex orthogonal matrices $U_{n-2},U_{n-3},\ldots,U_2$ to eliminate the superdiagonal entries in the $(n-3,n-2),(n-4,n-3),\ldots,(1,2)$ entry of the matrices $(U_{n-1} U_{n} C), (U_{n-2}U_{n-1}U C),\ldots, (U_3 \cdots U_{n-1} U_{n}C)$,  respectively.  We will denote the product $U_2U_3\cdots U_n$ of complex orthogonal transforms by $U$:
\begin{equation}
	U \eqdef U_2U_3\cdots U_n\,.
	\label{eq:Udef}
\end{equation}
Obviously, the matrix $U$ is also complex orthogonal (i.e. $U^T=U^{-1}$).
As all superdiagonal elements of $C$ are eliminated by $U$, the matrix $UC$ is lower triangular, and is given by the formula
\begin{equation}
	UC=UA+(Up)q^T.
	\label{eq:toB}
\end{equation}
We denote the matrix $UA$ by $B$, and the vector $Up$ by $\underline{p}$:
\begin{equation}
	B \eqdef UA\,, \quad \underline{p} \eqdef  Up\,.
\end{equation} 
Due to Lemma~\ref{lem:triang}, the upper Hessenberg part of the matrix $B$ is determined entirely by:
\begin{enumerate}
	\item The diagonal entries $\underline{d}_i=b_{i,i}$,  for $i=1,2,\ldots,n$,
	\item The subdiagonal entries $\underline{\gamma}_i=b_{i+1,i}$,  for $i=1,2,\ldots, n-1$,
	\item The vectors $\underline{p}, q$.
\end{enumerate}
Similarly,  the four vectors $\underline{d}, \underline{\gamma}, \underline{p}$ and $q$ are the generators of (the upper Hessenberg part of) $B$.  

Next,  we multiply $UC$ by $U^T$ on the right to bring the lower triangular $UC$ back to a lower Hessenberg form. Thus we obtain $UCU^T$ via the formula
\begin{equation}
	UCU^T = UAU^T + Up (Uq)^T\,.
	\label{eq:UCUT}
\end{equation}
We denote the matrix $UAU^T$ by $\underline{A}$, and the vector $Uq$ by $\underline{q}$:
\begin{equation}
	\underline{A}=BU^T \eqdef UAU^T, \quad \underline{q}\eqdef Uq .
	\label{eq:toAbar}
\end{equation}
This completes one iteration of the complex orthogonal QR algorithm.
Clearly, the matrix $UCU^T$ is similar to $C$ as $U$ is complex orthogonal, so they have the same eigenvalues. 

It should be observed that the matrix $UCU^T = \underline{A} + \underline{p}\underline{q}^T$ in (\ref{eq:UCUT}) is still a sum of a complex symmetric matrix and a rank-1 update. Moreover, multiplying $U^T$ on the right of the lower triangular matrix $B$ brings it back to the lower Hessenberg form, since $U^T$ is the product of $U_{n}^T, U_{n-1}^T, \ldots, U_{2}^T$ (See (\ref{eq:Udef}) above), where $U^T_k$ only rotates column $k$ and column $k-1$ of $B$.
As a result, the matrix $UCU^T$ is lower Hessenberg, thus still in the class $\mathcal{F}$ in (\ref{eq:A}). Again, due to Lemma~\ref{lem:tridiag}, the matrix $\underline{A}$ can be represented by its four generators.

We will see, in the next two sections, that the generators of $B$ can be obtained directly from generators of $A$, and the generators of $\underline{A}$ can be obtained from those of $B$.  Consequently, it suffices to only rotate generators of $A$ from (\ref{eq:A}) to (\ref{eq:toB}), followed by only rotating generators of $B$ to go from (\ref{eq:toB}) to (\ref{eq:toAbar}). Thus, each QR iteration only costs $O(n)$ operations. Moreover, since $UCU^T$ remains in the class $\mathcal{F}$, the process can be iterated to find all $n$ eigenvalues of the matrix $A$ in $O(n^2)$ operations, provided all complex orthogonal transforms $Q_k$ are defined (See Lemma~\ref{lem:rot}). The next two sections contain the details of the complex orthogonal QR algorithm outlined above.

\subsection{Eliminating the superdiagonal \label{sec:qr:super}(Algorithm~\ref{alg:elim})}  




In this section, we describe how the QR algorithm eliminates all superdiagonal elements of the matrix $A+pq^T$,  proceeding from (\ref{eq:A}) to (\ref{eq:toB}).  Suppose we have already eliminated the superdiagonal elements in the positions $(n-1,n),  (n-2,n-1),\ldots,(k,k+1)$.  We denote the vector $U_{k+1}U_{k+2}\cdots U_n p$ and the matrix $U_{k+1}U_{k+2}\cdots U_n A$, for $k=1,2,\ldots,n-1$, by $p^{(k+1)}$  and $B^{(k+1)}$ respectively:
\begin{equation}
	p^{(k+1)}\eqdef U_{k+1}U_{k+2}\cdots U_n p \quad \mathrm{and}\quad \quad B^{(k+1)}\eqdef U_{k+1}U_{k+2}\cdots U_n A\,.
	\label{eq:pBnote}
\end{equation}
For convenience, we define $p^{(n+1)}=p$ and $B^{(n+1)}=A$. Thus the matrix we are working with is given by
\begin{equation}
	B^{(k+1)} + p^{(k+1)}q^T\,.
	\label{eq:Bk1}
\end{equation}
Since we have rotated a part of the superdiagonal elements of $A$ in (\ref{eq:A}) to produce part of the subdiagonal elements of $B$ in (\ref{eq:toB}),
the upper Hessenberg part of $B^{(k+1)}$ is represented by the generators
\begin{enumerate}
	\item The diagonal elements $d_i^{(k+1)}=b^{(k+1)}_{i,i}$,  for $i=1,2,\ldots,n$,
	\item The superdiagonal elements $\beta^{(k+1)}_i =  b^{(k+1)}_{i,i+1}$,  for $i=1,2,\ldots,k-1$,
	\item The subdiagonal elements $\gamma^{(k+1)}_i = b^{(k+1)}_{i+1,i}$,  for $i=1,2,\ldots,n-1$,
	\item The vectors $p^{(k+1)}$ and $q$ from which the remaining elements in the upper Hessenberg part are inferred.
\end{enumerate}

\begin{figure}
  \renewcommand*{\arraystretch}{2.0}
\centering
\begin{minipage}{0.5\textwidth}
\begin{align*}
&\hspace*{-4em}
\left(
  \begin{array}{ccccccc}
\ddots &  \gamma_{k-3}^{(k+1)} &  d_{k-2}^{\;(k+1)} & 
  \beta_{k-2}^{\,(k+1)} & - p_{k-2}^{\,(k+1)} q_k & -
  p_{k-2}^{\,(k+1)} q_{k+1} & \cdots \\
\hline
\cdots & \times & \gamma_{k-2}^{(k+1)} &  d_{k-1}^{\;(k+1)} &
  \beta_{k-1}^{\,(k+1)} & - p_{k-1}^{\,(k+1)} q_{k+1} & \cdots \\
\cdots & \times &  b_{k,k-2}^{\,(k+1)} & \hat \gamma_{k-1}^{(k+1)} & 
  d_k^{\;(k+1)} & - p_k^{\,(k+1)} q_{k+1} & \cdots \\
\hline
\cdots & \times & \times & \times &  \gamma_{k}^{(k+1)} &
  d_{k+1}^{\;(k+1)} & \ddots
  \end{array}
\right)
\end{align*}
\end{minipage}
  \caption{The $(k-1)$-th and $k$-th rows of $B^{(k+1)}$, represented by
  its generators.  \label{fig:bhat}}
\end{figure}

First we compute the complex orthogonal  matrix $Q_k$ via Lemma~\ref{lem:rot} that eliminates the superdiagonal element in the $(k-1,k)$ position of the matrix (\ref{eq:Bk1}), given by the formula,
\begin{equation}
B^{(k+1)}_{k-1,k}+ (p^{(k+1)}q^T)_{k-1,k} = \beta^{(k+1)}_{k-1} + p^{(k+1)}_{k-1}q_k,	
\label{eq:superelim}
\end{equation}
by rotating row $k$ and row $k-1$ (See Figure~\ref{fig:bhat}).
Next,  we apply the complex orthogonal matrix $Q_k$ separately to the generators of $B^{(k+1)}$ and to the vector $p^{(k+1)}$.  Since we are only interested in computing the upper Hessenberg part of $B^{(k)}$ (See Lemma~\ref{lem:triang}),  we only need to update the diagonal element $d^{(k+1)}_k,  d^{(k+1)}_{k-1}$ in the $(k,k)$ and $(k-1,k-1)$ position and the subdiagonal element $\gamma^{(k+1)}_{k-1}$ and $\gamma^{(k+1)}_{k-2}$ in the $(k,k-1)$ and $(k-1,k-2)$ position. They are straightforward to update except for $\gamma^{(k+1)}_{k-2}$ since updating it requires the sub-subdiagonal element $b^{(k+1)}_{k,k-2}$ in the $(k,k-2)$ plane. This element can be inferred via the following observation.
 

Suppose we multiply the matrix (\ref{eq:Bk1}) by the rotations $U_n^T, U_{n-1}^T,  \cdots,  U^T_{k+1}$ from the right, and the result is given by the formula
\begin{equation}
	B^{(k+1)}U_n^T U_{n-1}^T \cdots U^T_{k+1} + p^{(k+1)}(U_{k+1}U_{k+2}\cdots U_n q)^T.  
	\label{eq:middle}
\end{equation}   
Since multiplying $U^T_{j}$ from the right only rotates column $j$ and $j-1$ of a matrix,  applying $U_n^T, U_{n-1}^T,  \cdots,  U^T_{k+1}$ rotates the matrix  (\ref{eq:Bk1}) back to a lower Hessenberg one.  For the same reason,  the desired element $b^{(k+1)}_{k,k-2}$ is not affected by the sequence of transforms $U_n^T U_{n-1}^T \cdots U^T_{k+1}$. In other words, $b^{(k+1)}_{k,k-2}$ is also the $(k,k-2)$ entry of $B^{(k+1)}U_n^T U_{n-1}^T \cdots U^T_{k+1}$:
\begin{equation}
	 \bb{B^{(k+1)}U_n^T U_{n-1}^T \cdots U^T_{k+1}}_{k,k-2} = b^{(k+1)}_{k,k-2}\,.
	\label{eq:ssinfer}
\end{equation}
Moreover, the matrix $B^{(k+1)} U_n^T U_{n-1}^T \cdots U^T_{k+1}$ is complex symmetric since $B^{(k+1)} = U_{k+1} U_{k+2}\cdots U_n A$ (See (\ref{eq:pBnote} above)) and $A$ is complex symmetric.   As a result,  
the matrix in (\ref{eq:middle}) is a complex symmetric matrix with a rank-1 update.  Due to Lemma~\ref{lem:tridiag},  the sub-subdiagonal element in (\ref{eq:ssinfer}) of the complex symmetric part can be inferred from the $(k-2,k)$ entry of the rank-1 part $p^{(k+1)}(U_{k+1}U_{k+2}\cdots U_n q)^T$. This observation shows that it is necessary to compute the following vector, denoted by $\tilde{q}^{(k+1)} $:
\begin{equation}
	\tilde{q}^{(k+1)} \eqdef U_{k+1}U_{k+2}\cdots U_n q.
	\label{eq:qt}
\end{equation}  
Then the sub-subdiagonal $b^{(k+1)}_{k,k-2}$ is computed by the formula
\begin{equation}
	b^{(k+1)}_{k,k-2} = -\tilde{q}^{(k+1)}_k p^{(k+1)}_{k-2}\,.
	\label{eq:infer}
\end{equation}
(See Line~\ref{alg:elim:qtil} of Algorithm~\ref{alg:elim}.)

After obtaining $b^{(k+1)}_{k,k-2}$,  we update the remaining generators affected by the elimination of the superdiagonal (\ref{eq:superelim}). 
The element in the $(k-1,k-2)$ position of $B^{(k+1)}$, represented by $\gamma^{(k+1)}_{k-2}$,  is updated in Line~\ref{alg:elim:subxsubsub} of Algorithm~\ref{alg:elim}.   Next,  the elements in the $(k-1,k-1)$ and $(k,k-1)$ positions of $B^{(k+1)}$,  represented by $d_{k-1}^{(k-1)}$ and $\gamma_{k-1}^{(k+1)}$ respectively,  are updated in a straightforward way in Line~\ref{alg:elim:diagxsub}.  Finally,  the elements in the $(k-1,k)$ and $(k,k)$ positions of $B^{(k+1)}$,  represented by $\beta_{k-1}^{(k+1)}$ and $d_k^{(k+1)}$ respectively,  are updated in Line~\ref{alg:elim:supxdiag},  and the vector $p^{(k+1)}$ is rotated in Line \ref{alg:elim:p}.

\subsubsection{Correction for large $\norm{pq^T}$ \label{sec:qr:correction}}
As discussed in Remark~\ref{rmk:large}, the size of the vector $q$ in generalized colleague matrices 
can be orders of magnitude larger than those in the complex symmetric part $A$. These elements are ``mixed'' when the superdiagonal is eliminated in Section~\ref{sec:qr:super}. When it is done in finite-precision arithmetic, large rounding errors due to the large size of $q$ can occur -- they can be as large as the size of elements in $A$, completely destroying the precision of eigenvalues to be computed. Fortunately, the representation via generators permits rounding errors to be corrected even when large $q$ is present. The corrections are as follows.

In Section~\ref{sec:qr:super}, we described in some detail the process of elimination of the $(k-1,k)$ entry in (\ref{eq:superelim}) of the matrix $B^{(k+1)} +p^{(k+1)}q^T$ in (\ref{eq:Bk1}) by the matrix $Q_k$.  
After rotating all the generators, we update $\beta^{(k+1)}_{k-1}$ and $p^{(k+1)}_{k-1}$ by $\beta^{(k)}_{k-1}$ and $p^{(k)}_{k-1}$ respectively. 
Thus, in exact arithmetic, the updated $(k-1,k)$ entry becomes zero. More explicitly, we have
\begin{equation}
	B^{(k)}_{k-1,k} +(p^{(k)}q^T)_{k-1,k}  = \beta^{(k)}_{k-1} + p^{(k)}_{k-1}q_k = 0\,.
	\label{eq:zero}
\end{equation}
Thus we in principle only need to keep $p^{(k)}_{k-1}$ to infer $\beta^{(k)}_{k-1}$. (This is why $\beta^{(k)}_{k-1}$ is not part of the generators, as can be seen from the representation of $B^{(k+1)}_{k-1,k}$ below (\ref{eq:Bk1}).)
However, when rounding error are considered and elements in $q$ are large, inferring $\beta^{(k)}_{k-1}$ from the computed value $\hat{p}^{(k)}_{k-1}q_k$ of $p^{(k)}_{k-1}q_k$ can lead to errors much greater than the directly computed value $\hat{\beta}^{(k)}_{k-1}$. This can be understood by first observing that $\beta^{(k)}_{k-1}$ and $p^{(k)}_{k-1}q_k$ are obtained by applying $Q_k$ to the vectors
\begin{equation}
	(\beta^{(k+1)}_{k-1},d^{(k+1)}_{k})^T,\quad (p^{(k+1)}_{k-1}q_k, p^{(k+1)}_{k}q_k)^T
\end{equation}
respectively. Furthermore, the rounding errors are of the size the machine epsilon $u$ times the norms of these vectors. 
As a result, the computed errors in $\hat{\beta}^{(k)}_{k-1}$ and $p^{(k)}_{k-1}q_k$ are of the order 
\begin{equation}
	\sqrt{|\beta^{(k+1)}_{k-1}|^2 + |d^{(k+1)}_{k}|^2}\norm{Q_k}u,\quad \sqrt{|p^{(k+1)}_{k-1}q_k|^2 + | p^{(k+1)}_{k}q_k|^2}\norm{Q_k}u
\end{equation}
respectively, where $\norm{Q_k}$ is the size of the complex orthogonal transform. (The norm $\norm{Q_k}$ is 1 if it is unitary, i.e. a true rotation.)
We observe that inferring $\hat{\beta}^{(k)}_{k-1}$ from $\hat{p}^{(k)}_{k-1}q_k$ leads to larger computed error when
\begin{equation}
	 |{p}^{(k+1)}_{k-1}q_k|^2 + | {p}^{(k+1)}_{k}q_k|^2 > |{\beta}^{(k+1)}_{k-1}|^2 + |{d}^{(k+1)}_{k}|^2\,,
\end{equation}
so we apply a correction based on (\ref{eq:zero}) by setting the computed $\hat{p}_{k-1}^{(k)}$ to be
\begin{equation}
	\hat{p}_{k-1}^{(k)} = -\hat{\beta}^{(k)}_{k-1}/q_{k}.
	\label{eq:correction}
\end{equation}
(See Line~\ref{alg:elim:corr} of Algorithm \ref{alg:elim}.) On the other hand, when
\begin{equation}
	|{p}^{(k+1)}_{k-1}q_k|^2 + | {p}^{(k+1)}_{k}q_k|^2 \le |{\beta}^{(k+1)}_{k-1}|^2 + |{d}^{(k+1)}_{k}|^2,
\end{equation}
the correction in (\ref{eq:correction}) is not necessary since the error in $\hat{p}_{k-1}^{(k)}q_{k}$ will be smaller than the error in $\hat{\beta}^{(k)}_{k-1}$. The correction described above is essential in achieving structured backward stability for the QR algorithm with SU(2) rotations in \cite{serkh2021} for classical colleague matrices. We refer readers to \cite{serkh2021} for a detailed discussion on the correction's impact on the stability of the algorithm. 

This process of eliminating the superdiagonal elements can be repeated, until the upper Hessenberg part of the matrix $B=U_2 U_3 \cdots U_n A$ is obtained, together with the vector $\underline{p}=U_2 U_3 \cdots U_n p$.    



\subsection{Rotating back to Hessenberg form  \label{sec:qr:sub} (Algorithm~\ref{alg:rotback})}
 In this section, we describe how our QR algorithm rotates the triangular matrix $UC$ in \eqref{toB} back to lower Hessenberg form $UC U^T$ in \eqref{toAbar}.  Given the matrix $B$ and vectors $\underline{p}$ and $q$ from the preceding section, the sum $B+\underline{p}q^T$ is lower triangular.  
Suppose that we have already applied the rotation matrix $U_n^T,  U^T_{n-1},  \cdots,  U^T_{k+1}$ to right of $B$ and $q^T$. We denote the vector $U_{k+1}U_{k+2}\cdots U_n q$ by $q^{(k+1)}$ and the matrix $BU^T_n U^T_{n-1}\cdots U^T_{k+1}$ by $A^{(k+1)}$:
\begin{equation}
	q^{(k+1)} \eqdef U_{k+1}U_{k+2}\cdots U_n q, \quad \mathrm{and}\quad A^{(k+1)} \eqdef BU^T_n U^T_{n-1}\cdots U^T_{k+1}.
\end{equation}
For convenience, we define $q^{(n+1)}=q $ and $A^{(n+1)}=B$.  Immediately after applying $U^T_{k+1}$,
we have rotated part of the subdiagonal element of $B$ in \eqref{toB} to produce part of the superdiagonal elements of $\underline{A}$ in \eqref{toAbar},  
so the upper Hessenberg part of  $A^{(k+1)}$ is represented by the generators
\begin{enumerate}
	\item  The  diagonal entries $d^{(k+1)} = a_{i,i}^{(k+1)}$,  for $i=1,2,\ldots,n$,
	\item The superdiagonal entries $\beta^{(k+1)}_{i} = a_{i,i+1}^{(k+1)}$ for $i=k,k+1,\ldots,  n-1$,
    \item The subdiagonal entries $\gamma^{(k+1)}_{i} = a_{i+1,i}^{(k+1)}$ for $i=1,2,\ldots, k-1$,
	\item The vectors $\underline{p}$ and $q^{(k+1)}$,  from which the remaining elements in the upper triangular part are inferred.
\end{enumerate}

\begin{figure}
  \renewcommand*{\arraystretch}{1.7}
\centering
\begin{minipage}{0.5\textwidth}
\begin{align*}
&\hspace*{-4em}
\left(
  \begin{array}{c|cc|c}
\ddots & \vdots & \vdots & \vdots \\
 d_{k-2}^{\;(k+1)} & -p_{k-2} q_{k-1}^{\,(k+1)} &
  -p_{k-2} q_{k}^{\,(k+1)} & -p_{k-2}  q_{k+1}^{\,(k+1)} \\
\gamma_{k-2}^{(k+1)}  &  d_{k-1}^{\;(k+1)} & -p_{k-1}  q_{k}^{\,(k+1)} &
  -p_{k-1}  q_{k+1}^{\,(k+1)} \\
\times & \gamma_{k-1}^{(k+1)} &  d_{k}^{\;(k+1)}  &  \beta_{k}^{\,(k+1)}  \\
\times & \times & \times &  d_{k+1}^{\;(k+1)}  \\
\vdots & \vdots & \vdots & \ddots 
  \end{array}
\right)
\end{align*}
\end{minipage}
  \caption{The $(k-1)$-th and $k$-th columns of $ A^{(k+1)}$,
  represented by its generators.  \label{fig:ahat}}
\end{figure}

To multiply the matrix $A^{(k+1)} + \underline{p}(q^{(k+1)})^T$ by $U_k^T$ from the right,  we apply the complex orthogonal rotation matrix $Q_k^T$ separately to the generators of $A^{(k+1)}$ and to the vector $q^{(k+1)}$.
We start by rotating the diagonal and superdiagonal elements in the $(k-1,k-1)$ and $(k-1,k)$ positions of $A^{(k+1)}$,  represented by $d_{k-1}^{(k+1)}$ and $-\underline{p}_{k-1}q_k^{(k+1)}$ respectively,  in Line~\ref{alg:rotback:diagsup} of Algorithm \ref{alg:rotback}, saving the superdiagonal element in $\beta_{k-1}^{(k)}$ (See 
Figure~\ref{fig:ahat}).  
Next, we rotate the elements in the $(k,k-1)$ and $(k,k)$ positions,  represented by $\gamma^{(k+1)}_{k-1}$
and $d_k^{(k+1)}$ respectively,  in a straightforward way in Line~\ref{alg:rotback:subdiag}; since we are only interested in computing the upper triangular part of $A^{(k)}$,  we only update the diagonal entry.  Finally,  we rotate the vector $q^{(k+1)}$ in Line~\ref{alg:rotback:q}. This process of applying the complex orthogonal transforms is repeated until the matrix $\underline{A}$ and the vector $\underline{q}$ in (\ref{eq:toAbar}) are obtained.

\subsection{The QR algorithm with explicit shifts (Algorithm \ref{alg:qrshift})}

The elimination of the superdiagonal described in Section~\ref{sec:qr:super} (Algorithm \ref{alg:elim}),  followed by transforming back to Hessenberg form described in Section~\ref{sec:qr:sub} (Algorithm \ref{alg:rotback}),  makes up a single iteration in the complex orthogonal QR algorithm for finding eigenvalues of $A+pq^T$.  In principle, it can be iterated to find all eigenvalues. However,  unlike the algorithms in \cite{serkh2021},  complex orthogonal rotations in our algorithm can get large occasionally. (A typical size distribution of the rotations  can be found in Section~\ref{sec:ex_poly}.) When no shift is applied in QR iterations,  the linear convergence rate of eigenvalues will generally require  an unreasonable number of iterations to reach a acceptable precision. Although large size rotations are rare, they do sometimes accumulate to sizable numerical error when the iteration number is large. This can cause breakdowns of the complex orthogonal QR algorithm, and such breakdowns have been observed in numerical experiments. As a result, the introduction of shifts is essential in order to accelerate convergence, reducing the iteration number and avoiding breakdowns. In this paper, we only provide the QR algorithm with explicit Wilkinson shifts in Algorithm \ref{alg:qrshift}. No breakdowns of Algorithm \ref{alg:qrshift} have been observed once the shifts are introduced; the analysis of its numerical stability, however, is still under vigorous investigation.  All eigenvalues in this paper are found by this shifted version of complex orthogonal QR.

\section{Description of rootfinding algorithms \label{sec:algorithm}}
In this section,  we describe algorithms for finding all roots of a given complex analytic function over a square domain. Section~\ref{sec:precom} contains the precomputation of a polynomial basis satisfying a three-term recurrence on a square domain (Algorithm~\ref{alg:precom}).
Section~\ref{sec:nonadap} contains the non-adaptive rootfinding algorithm (Algorithm~\ref{alg:nadap_root}) on square domains, followed by the adaptive version (Algorithm~\ref{alg:adap_root}) in Section~\ref{sec:adap}. 

The non-adaptive rootfinding algorithm's inputs are a square domain $D\subset\C$, specified by its center $z_0$ and side length $2l$, and a function $f$ analytic in $D$,
whose roots in $D$ are to be found, and two accuracies $\epsilon_{\rm exp}$ and $\epsilon_{\rm eig}$. The first $\epsilon_{\rm exp}$ controls the approximation accuracy of the input function $f$ by polynomials on the boundary $\partial D$; the second $\epsilon_{\rm eig}$ specifies the accuracy of eigenvalues of the colleague matrix found by the complex orthogonal QR.
In order to robustly catch all roots very near the boundary $\partial D$, a small constant $\delta>0$ is specified, so that
all roots within a slightly extended square of side length $2l(1+\delta)$ are retained and they are the output of the algorithm.  We refer to this slightly expanded square as the $\delta$-extended domain of $D$.

The non-adaptive rootfinding algorithm consists of two stages. First, an approximating polynomial in the precomputed basis of the given function $f$ is computed via least squares. From the coefficients of the approximating polynomial, a generalized colleague matrix is formed. Second, the complex orthogonal QR is applied to compute the eigenvalues, which are also the roots of the approximating polynomial (Theorem~\ref{thm:colleague}). The roots of the polynomial inside the domain $D$ are identified as the computed roots of the function $f$.

The adaptive version of the rootfinding algorithm takes an additional input parameter $n_{\rm exp }$, the order of polynomial expansions for approximating $f$. It can be chosen to be relatively small so the expansion accuracy $\epsilon_{\rm exp}$ does not have to be achieved on the domain $D$. When this happens, the algorithm continues to divide the square $D$ adaptively into smaller squares until the accuracy $\epsilon_{\rm exp}$ is achieved on those smaller squares. Then the non-adaptive algorithm is invoked on each smaller square obtained from the subdivision. After removing redundant roots near boundaries of neighboring squares, all the remaining roots are collected as the output of the adaptive algorithm.




\subsection{Basis precomputation \label{sec:precom} (Algorithm~\ref{alg:precom})}
We follow the description in Section~\ref{sec:three} to construct a reasonably well-conditioned basis of order $n$ on the boundary $\partial \Omega$ of the square $\Omega$, shown in Figure~\ref{fig:square}, centered at the origin with side length 2. For convenience, we choose the points where the polynomials are constructed to be Gaussian nodes on each side of the square $\Omega$. (See Section~\ref{sec:data}.) 
In the non-adaptive case, the order $n$ is chosen to be sufficiently large so that the expansion accuracy $\epsilon_{\rm exp}$ is achieved; in the adaptive case, $n$ is specified by the user.
\begin{enumerate}
	\item Choose the number of Gaussian nodes $k$ ($k\ge n/2$) on each side the square. Generate points $z_1 ,z_2,\ldots, z_m$, consisting of all Gaussian nodes from the four sides (See Figure~\ref{fig:square}), where the total number $m=4k$. 
	Generate $m$ random weights $w_1,w_2,\ldots,w_m$, uniformly drawn from $[0,1]$ to define a complex inner product in the form of (\ref{eq:com_prod}).
	\item Perform the complex orthogonalization process (See Lemma~\ref{lem:ttr_poly}) to obtain a polynomial basis $P_0, P_1, \ldots,  P_n$ of order $n$ on the boundary $\partial \Omega$ at the chosen points $z_1,  z_2, \ldots,z_m$, and the two vectors $\alpha,\beta\in \C^n$. The polynomials in $\cb{P_j}_{j=1}^n$ satisfies the three-term recurrence relation in the form of (\ref{eq:three_thm}), defined by the vectors $\alpha,\beta$. 
	\item Generate Gaussian weights $\tilde{w}_1, \tilde{w}_2,\ldots,\tilde{w}_m$ corresponding to Gaussian nodes $z_1,  z_2, \ldots,z_m$ on the four sides. Form the $m$ by $n+1$ basis matrix $G$ given by the formula
	\begin{equation}
		G =\bb{\mat{
		\sqrt{\tilde{w}_1} P_0(z_1) & \sqrt{\tilde{w}_1} P_1(z_1) & \ldots & \sqrt{\tilde{w}_1} P_n(z_1) \\
		\sqrt{\tilde{w}_2} P_0(z_2) & \sqrt{\tilde{w}_2} P_1(z_2) & \ldots & \sqrt{\tilde{w}_2} P_n(z_2) \\
		\vdots & \vdots & \ddots & \vdots \\
		\sqrt{\tilde{w}_m} P_0(z_m) & \sqrt{\tilde{w}_m} P_1(z_m) & \ldots & \sqrt{\tilde{w}_m} P_n(z_m)
		}}\,.
		\label{eq:gmat}
	\end{equation}
	\item Initialize for solving least squares involving the matrix $G$ in (\ref{eq:gmat}), by computing
	the reduced QR factorization of $G$
	\begin{equation*}
		G=QR\,,
	\end{equation*}
	where $Q\in\C^{m\times(n+1)}$ and $R\in\C^{(n+1)\times(n+1)}$.
\end{enumerate}
The precomputation is relatively inexpensive and only needs to be done once -- it can be done on the fly or precomputed in advance and stored. 
The vectors $\alpha,\beta$ and the matrices $Q, R$ are the only information required in our algorithms for rootfinding in a square domain.

\begin{remark}
	There is nothing special about the choice of QR above for solving least squares. Any standard method can be applied.
\end{remark}

\begin{remark}
	In the complex orthogonalization process, the constructed vector $q_{i}$ will automatically be complex orthogonal to any $q_j$ for $j<i-2$ in exact arithmetic due to the complex symmetry of the matrix $Z$. In practice, when $q_i$ is computed, reorthogonalizing $q_i$  against all preceding $q_j$ for $j<i$ is desirable to increase numerical stability. After Line~\ref{alg:precom:ortho} of Algorithm~\ref{alg:precom}, we reorthogonalize $v$ via the formula 
	\begin{equation}
		v\leftarrow v - [v,q_{i}] q_{i} - [v,q_{i-1}] q_{i-1} - [v,q_{i-2}]q_{i-2} \cdots - [v,q_{0}]q_{0}\,.
		\label{eq:reortho}
	\end{equation}
	Doing so will ensure the complex orthonormality of the computed basis $q_0,q_1,q_2,\ldots,q_n$ holds to near full accuracy. The reorthogonalization in (\ref{eq:reortho}) can be repeated if necessary to increase the numerical stability further.
\end{remark}


\subsection{Non-adaptive rootfinding \label{sec:nonadap} (Algorithm~\ref{alg:nadap_root})}

We first perform the precompution described Section~\ref{sec:precom}, obtaining a polynomial basis $\cb{P_j}_{j=0}^n$ of order $n$ on the boundary $\partial \Omega$, the vectors $\alpha, \beta$ defining the three-term recurrence, and the reduced QR of the basis matrix $G$ in (\ref{eq:gmat}).

\subsubsection*{Stage 1: construction of colleague matrices}
Since the basis $\cb{P_j}_{j=0}^n$ is precomputed on the square $\Omega$ (See Section~\ref{sec:precom}),
we map the input square $D$ to $\Omega$ via a translation and a scaling transform.
Thus we transform the function $f$ in $D$ accordingly into its translated and scaled version $\tilde{f}$ in $\Omega$, so that we find roots of $f$ in $D$ by first finding roots of $\tilde{f}$ in $\Omega$, followed by transforming those  roots back from $\Omega$ to $D$. To find roots of  $\tilde{f}$, we form an approximating polynomial $p$ of $\tilde{f}$, given by the formula
	\begin{equation}
		p(z)=\sum_{j=0}^n c_j P_{j}(z),
		\label{eq:pexp}
	\end{equation}
whose expansion coefficients are computed by solving a least-squares problem. 

\begin{enumerate}
	\item Translate and scale $f$ via the formula
	\begin{equation}
		\tilde{f}(z) = f(l z + z_0)\,.
	\end{equation}
	\item Form a vector $g\in \C^{m}$ given by the formula 
	\begin{equation}
		g = \bb{\mat{\sqrt{\tilde{w}_1} \tilde{f}(z_1) \\ \sqrt{\tilde{w}_2} \tilde{f}(z_2) \\ \vdots\\\sqrt{\tilde{w}_m} \tilde{f}(z_m)}}\,.
	\end{equation} 
	\item 
	Compute the vector $c=\bb{c_0,c_1,\ldots,c_n}^T\in \C^{n+1}$, containing the expansion coefficients in (\ref{eq:pexp}),
	via the formula
	\begin{equation}
		c= R^{-1}Q^*g \,,
	\end{equation}
	which is the least-squares solution to the linear system $Gc=g$. This procedure takes $O(mn)$ operations. 
	\item Estimate the expansion error by ${\abs{c_n}}/{\norm{c}}$. 
	If ${\abs{c_n}}/{\norm{c}}$ is smaller than the given expansion accuracy $\epsilon_{\rm exp}$, accept the expansion and the coefficient vector $c$. Otherwise, repeat the above procedures with a larger $n$. 
\end{enumerate}
The colleague matrix $C=A+e_n q^T$ in (\ref{eq:colleague}) in principle can be formed explicitly -- the complex symmetric $A$ can be formed by elements in the vectors $\alpha, \beta$, and the vector $q$ can be obtained from the coefficient vector $c$ together with the last element of $\beta$.
However, this is unnecessary since we only rotate generators of $C$ in the complex orthogonal QR algorithm (See Section~\ref{sec:QR}). The generators of $C$ here are vectors $\alpha, \beta, e_n$ and $q$ and they will be the inputs of the complex orthogonal QR\,. 

\begin{remark}
	It should be observed that the approximating polynomial $p$ of the function $\tilde{f}$ in (\ref{eq:pexp}) is only constructed on the boundary of $\partial \Omega$ since the points $z_1, z_2, \ldots, z_m$ used in forming the basis matrix $G$ in (\ref{eq:gmat}) are all on $\partial \Omega$.
However, due to the maximum principle (Theorem~\ref{thm:max}), the largest approximation error $\mathrm{max}_{z\in \Omega}|{\tilde{f}(z)-p(z)}|$ is attained on the boundary $\partial \Omega$ since the difference $\tilde{f}-p$ is analytic in $\Omega$. ($\tilde{f}$ is analytic in $\Omega$ since $f$ is analytic in $D$.) Thus the absolute error of the approximating polynomial $p$ is smaller in $\Omega$. As a result, once the approximation on the boundary $\partial \Omega$ achieves the required accuracy, it holds over the entire $\Omega$ automatically.

\end{remark}

\begin{remark}
	When the expansion accuracy $\epsilon_{\rm exp}$ is not satisfied for an initial choice of $n$, it is necessary to recompute all basis for a larger order $n$. This can be avoided by choosing a generously large order $n$ in the precomputation stage. 
	 This could be undesirable since the number  $m$ of Gaussian nodes needed on the boundary also grows accordingly.
	When a very large $n$ is needed, it is recommended to switch to the adaptive version (Algorithm~\ref{alg:adap_root}) with a reasonable $n$ for better efficiency and accuracy. 
\end{remark}

\subsubsection*{Stage 2: rootfinding by colleague matrices}
Due to Theorem~\ref{thm:colleague}, we find the roots of $p$ by finding eigenvalues of the colleague matrix $C$. 
Eigenvalues of $C$ are computed by the complex orthogonal QR algorithm in Section~\ref{sec:QR}. 
We only keep roots of $p$ within $\Omega$ or very near the boundary $\partial \Omega$ specified by the constant $\delta$. These roots are identified as those of the transformed function $\tilde{f}$. Finally, we translate and scale roots of $\tilde{f}$ in $\Omega$ back so that we recover the roots of the input function $f$ in $D$.

\begin{enumerate}
	\item Compute the eigenvalues of the colleague matrix $C$, represented by its generators $\alpha,\beta,e_n$ and $q$, by the complex orthogonal QR (Algorithm~\ref{alg:qrshift}) in $O(n^2)$ operations. The eigenvalues/roots are labeled by $\tilde{r}_1, \tilde{r}_2,\ldots,\tilde{r}_n$.
	\item Only retain a root $\tilde{r}_j$ for $j=1,2,\ldots,n$, if it is inside the $\delta$-extended square of the domain $\Omega$:
	\begin{equation}
		\abs{\mathrm{Re}\,\tilde{r}_j} < 1+\delta\quad \mathrm{and}\quad \abs{\mathrm{Im}\,\tilde{r}_j} < 1+\delta.	
	\end{equation}

	\item Scale and translate all remaining roots back to the domain $D$ from $\Omega$
	\begin{equation}
		r_j = l\tilde{r}_j+z_0
	\end{equation}
	and all $r_j$ are identified as roots of the input function $f$ in $D$.
\end{enumerate}

\begin{remark}
	\label{rmk:newton}
	The robustness and accuracy of rootfinding can be improved by applying Newton's method 
	\begin{equation}
		r_j \leftarrow r_j - \frac{f(r_j)}{f'(r_j)}
	\end{equation}
	for one or two iterations after roots $r_i$ are located at little cost, if the values of $f$ and $f'$ are available inside $D$. 
	
\end{remark}

%
%


\subsection{Adaptive rootfinding \label{sec:adap} (Algorithm~\ref{alg:adap_root})}
In this section we describe an adaptive version of the rootfinding algorithm in Section \ref{sec:nonadap}. 
The adaptive version has identical inputs as the non-adaptive one, besides one additional input $n_{\rm exp}$ as the chosen expansion order. (Obviously, the input $n_{\rm exp}$ should be no smaller than the order of the precomputed basis $n$.) \hl{The choices of $n_{\rm exp}=20-30$ are tested to be robust for double precision rootfinding while $n_{\rm exp}=40-60$ is generally good for extended precision calculations.}
The algorithm has two stages. First, it divides the domain $D$ into smaller squares, on which the polynomial expansion of order $n_{\rm exp}$ converge to the pre-specified accuracy $\epsilon_{\rm exp}$. In the second stage, rootfinding is performed on these squares by procedures very similar to those in Algorithm~\ref{alg:nadap_root}, followed by removing duplicated roots near boundaries of neighboring squares.

\begin{figure}
	\centering
	\begin{subfigure}[t]{0.48\textwidth}
		\centering
		\includegraphics[width=\textwidth]{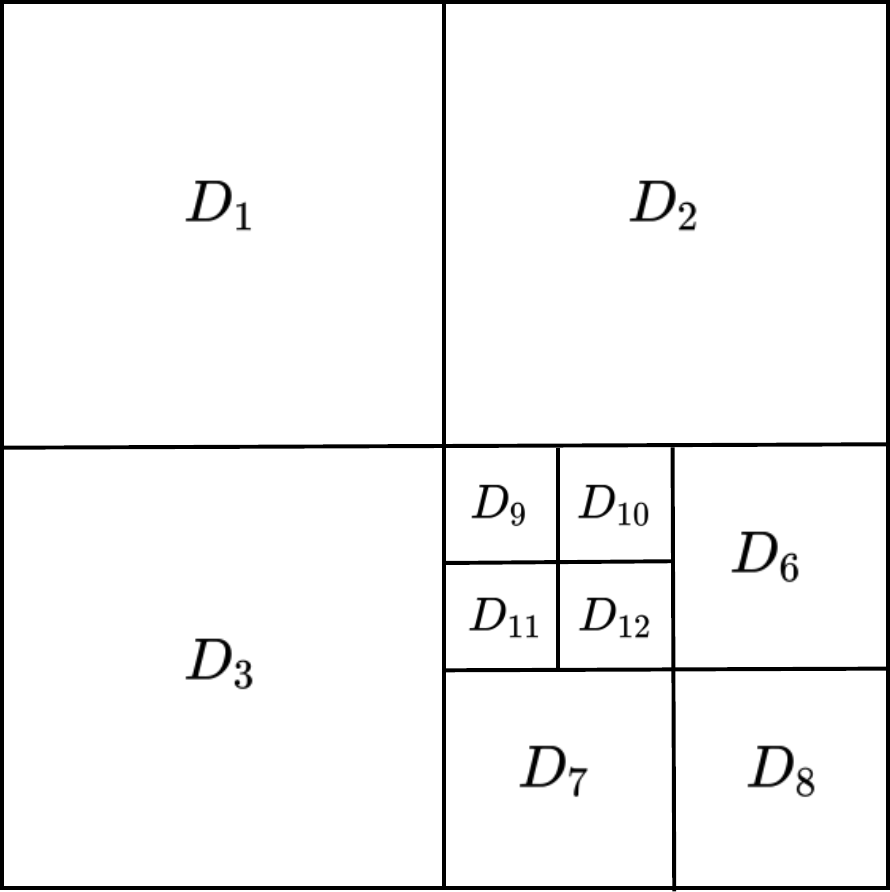}
	\end{subfigure}

	\caption{The input domain $D$, labeled $D_0$, is divided three times during an adaptive rootfinding by Algorithm~\ref{alg:adap_root}. Squares $D_4$ and $D_5$ are not shown since they are divided further; $D_4$ is divided in to $D_5, D_6, D_7, D_8$ and $D_5$ is divided into $D_9,D_{10},D_{11},D_{12}$. In each square shown here, rootfinding is done by finding eigenvalues of the colleague matrix. \label{fig:squares}}
\end{figure}

\subsubsection*{Stage 1: adaptive domain division}
In this stage, we recursively divide the input domain $D$ until the order $n_{\rm exp}$ expansion achieves the specified accuracy $\epsilon_{\rm exp}$. 
For convenience, we label the input domain $D$ as $D_0$. Here we illustrate this process by describing the case depicted in Figure~\ref{fig:squares} where $D$ is divided three times. Initially, we compute the order $n_{\rm exp}$ expansion on input domain $D_0$, as in Step 2 and 3 in Section~\ref{sec:nonadap}, and the accuracy is not reached. We divide $D_0$ into four identical squares $D_1,D_2,D_3,D_4$. We repeat the process to these four squares and the expansion accuracy except on $D_4$. Then we continue to divide $D_4$ into four pieces $D_5,D_6,D_7,D_8$.
Again we check expansion accuracy on all newly generated squares $D_5-D_{8}$ and the accuracy requirement is satisfied for all squares except for $D_5$. Then we continue to divide $D_5$ into $D_9, D_{10}, D_{11}, D_{12}$.
We keep all expansion coefficient vectors $c^{(i)}\in \C^{n_{\rm exp}+1}$ for (implicitly) forming colleague matrices on domains $D_i$, where the expansion accuracy $\epsilon_{\rm exp}$ is reached. As a result, rootfinding in this example only happens on squares $D_1, D_4$ and $D_5 - D_{12}$, so in total 10 eigenvalue problems of size $n_{\rm exp}$ by $n_{\rm exp}$ are solved.

In the following description, we initialize $D_0$ as the input domain $D$, and $i=0$. We also keep track the total number of squares $k$ ever formed; since initially we only have the input domain, the number of squares $k=1$.

\begin{enumerate}
	\item For domain $D_i$ centered at $z_0^{(i)}$ with side length $2l^{(i)}$, translate and scale the input function $\restr{f}{D_i}$ restricted to $D_i$ to the square $\Omega$ via the formula
		\begin{equation}
		\tilde{f}(z) = \restr{f}{D_i}(l^{(i)} z + z^{(i)}_0)\,.
	\end{equation}
	\item Form a vector $g^{(i)}\in \C^{m}$ given by the formula
	\begin{equation}
		g^{(i)} = \bb{\mat{\sqrt{\tilde{w}_1} \tilde{f}(z_1) \\ \sqrt{\tilde{w}_2} \tilde{f}(z_2) \\ \vdots\\\sqrt{\tilde{w}_m} \tilde{f}(z_m)}}\,.
	\end{equation} 
	\item 
	Compute the vector $c^{(i)}\in \C^{n_{\rm exp}+1}$ containing the expansion coefficients
	by the formula
	\begin{equation}
				c= R^{-1}Q^*g^{(i)} \,,
	\end{equation}
	which is the least-squares solution to the linear system $Gc^{(i)}=g^{(i)}$. This procedure takes $O(mn_{\rm exp})$ operations. 
	
	\item Estimate the expansion error by ${\abs{c^{(i)}_{n_{\rm exp}}}}/{\norm{c^{(i)}}}$. 
	If ${\abs{c^{(i)}_{n_{\rm exp}}}}/{\norm{c^{(i)}}}$ is smaller than the given expansion accuracy $\epsilon_{\rm exp}$, accept the expansion and the coefficient vector $c^{(i)}$ and save it for constructing colleague matrices.
	If the accuracy $\epsilon_{\rm exp}$ is not reached, continue to divide domain $D_i$ into four square domains $D_{k+1}, D_{k+2}, D_{k+3}, D_{k+4}$ and increment the number of squares $k$ by $4$, i.e. $k\leftarrow k+4 $.
	\item If $i$ does not equal $k$, this means there are still squares where expansions may not converge. Increment $i$ by 1 and go to Step 1.
		
\end{enumerate}
After all necessary subdivisions, all expansions of $f$ clearly achieve the accuracy $\epsilon_{\rm exp}$. 
Since the coefficient vector $c^{(i)}$ is saved whenever the expansion converges in the divided domain $D_i$, the corresponding colleague matrix $C^{(i)}$ can be formed for rootfinding on $D_i$. As in the non-adaptive case, the matrix $C^{(i)}$ is not formed explicitly but is represented by its generators as inputs to the complex orthogonal QR.
\subsubsection*{Stage 2: rootfinding by colleague matrices}
We have successfully divided $D$ into smaller squares. On each of those domain $D_i$, where an approximating polynomial of order $n_{\rm exp}$ converges to the specified accuracy $\epsilon_{\rm exp}$, we find roots of $f$ in the $\delta$-extended domain of $D_i$ as in Stage 2 of the non-adaptive version in Section~\ref{sec:nonadap}.
Then we collect all roots on each smaller squares and remove duplicated roots near edges of neighboring squares. The remaining roots will be identified as roots of $f$ on the $\delta$-extended domain of the input domain $D$.

\begin{enumerate}
	\item For all $D_i$ in which the expansion converges to $\epsilon_{\rm exp}$, form the generator $q^{(i)}$ according to the coefficient vector $c^{(i)}$ and $\beta$. Compute the eigenvalues of the colleague matrix $C^{(i)}$, represented by its generators $\alpha,\beta,e_n$ and $q^{(i)}$, by the complex orthogonal QR (Algorithm~\ref{alg:qrshift}) in $O(n_{\rm exp}^2)$ operations. Label eigenvalues for each $D_i$ as $\tilde{r}_j^{(i)}$.
	\item For each problem on $D_i$, only keep  a root $\tilde{r}^{(i)}_j$ if it is inside of the square domain slightly extended from $\Omega$:
	\begin{equation}
		\abs{\mathrm{Re}\,\tilde{r}_j} < 1+\delta\quad \mathrm{and}\quad \abs{\mathrm{Im}\,\tilde{r}_j} < 1+\delta.	
	\end{equation}

	\item Scale and translate all remaining roots back to the domain $D_i$ from $\Omega$
	\begin{equation}
		r_j^{(i)} = l^{(i)}\tilde{r}_j^{(i)}+z_0^{(i)}
	\end{equation}
	and all $r_j^{(i)}$ left are identified as roots of the input function $\restr{f}{D_i}$ restricted to $D_i$.
	\item Collect all roots of $r_j^{(i)}$ for all domain $D_i$ where rootfinding is performed. Remove duplicated roots found in neighboring squares. The remaining roots are identified as roots of the input function $f$ in the domain $D$.
\end{enumerate}

\begin{remark}
	In the final step of Stage 2, duplicated roots from neighboring squares are removed. It should be observed that  multiple roots can be distinguished from duplicated roots by the accuracy they are computed.	When the roots are simple, our algorithm achieves full accuracy determined the machine epsilon $u$, especially paired with refinements by Newton's method (see Remark~\ref{rmk:newton}). On the other hand, roots with multiplicity $m$ are sensitive to perturbations and can only be computed to the $m$th root of machine epsilon $u^{1/m}$. 
Consequently, one can distinguish if close roots are duplicated or due to multiplicity by the number digits they agree. 
\label{rmk:mult}
\end{remark}

\begin{algorithm}[H]
\caption{
\label{alg:precom}
(Precomputation of basis matrix and three term-recurrence) 
\textbf{Inputs}: The algorithm accepts $m$ nodes $z_1,z_2,\ldots,z_m$ that are Gaussian nodes on each sides of the square domain $\Omega$, along with $m$ Gaussian weights $\tilde{w}_1,\tilde{w}_2,\ldots,\tilde{w}_m$ associated with those $m$ nodes.
It also accepts another set of $m$ real weights $w_1, w_2, \ldots, w_m \in \C$, drawn uniformly from $[0,1]$, for the complex inner product in \eqref{com_prod}.
\textbf{Outputs}: It returns as output the vectors $\alpha, \beta \in \C^n$ that define the three-term recurrence relation. It also returns the reduced QR factorization of $QR$ of the matrix $G$ in \eqref{gmat}.
}
\begin{algorithmic}[1]
\State Define $Z=\mathrm{diag}(z_1,z_2,\ldots,z_m)$ and $b=(1,1,\ldots,1)\in \C^m$
\State
Initialization for complex orthogonalization

$q_{-1}\leftarrow 0, \beta_0 \leftarrow 0, q_0 \gets b/\sb{b}_w$ 

\For{$i=0,\ldots,n-1$}	

	$v\gets Zq_{i} $
\State
	Compute $\alpha_{i+1}$: $\alpha_{i+1} \gets \sb{q_i,v}_w$
	
	$v\gets v - \beta_{i}q_{i-1} - \alpha_{i+1}q_{i}$ \label{alg:precom:ortho}
	
	\State
	Compute $\beta_{i+1}$: $\beta_{i+1} \gets [v]_w$
	\State	
	Compute $q_{i+1}$: $q_{i+1} \gets v/\beta_{i+1}$
\EndFor

\State Set the $i$th component of $q_j$ as the value of $P_j$ at $z_i$:
$P_i(z_j) = (q_i)_j$

\State
Form the basis matrix $G\in \C^{m\times(n+1)}$ for least squares:

$G_{ij} \gets \sqrt{\tilde{w_j}}P_i(z_j)$.

\State
Compute the reduced QR factorization of $G$:

$QR \gets G$ \label{alg:precom:svd}
	
	\end{algorithmic}
\end{algorithm}

\newpage

\begin{algorithm}[H]
\caption{
  \label{alg:elim}
(A single elimination of the superdiagonal) \textbf{Inputs:}
This algorithm accepts
as inputs two vectors $d$ and $\beta$ representing the diagonal and
superdiagonal, respectively, of an $n\times n$ complex symmetric matrix $A$, as well
as two vectors $p$ and $q$ of length $n$, where $A+pq^T$ is lower
Hessenberg.  
\textbf{Outputs:} It returns as its outputs the rotation matrices
$Q_2,Q_3,\ldots, Q_n \in \C^{2 \times 2}$ so that, letting $U_k \in
\C^{n\times n}$, $k=2,3,\ldots,n$, denote the matrices that rotate the
$(k-1,k)$-plane by $Q_k$, $U_2 U_3 \cdots U_n (A + pq^T)$ is lower
triangular. It also returns the vectors $\underline d$, $\underline \gamma$,
and $\underline p$, where $\underline d$ and $\underline \gamma$ represent
the diagonal and subdiagonal, respectively, of the matrix $U_2 U_3 \cdots
U_n A$, and $\underline p = U_2 U_3 \cdots U_n p$.}

\begin{algorithmic}[1]

\State Set $\gamma\gets \beta$, where $\gamma$ represents the subdiagonal.
\State Make a copy of $q$, setting $\tilde q\gets q$.

\For{$k=n,n-1,\ldots,2$}

  \State 
    \label{alg:elim:qk}
  Construct the $2\times 2$ rotation matrix complex symmetric $Q_k$ so that
  \begin{align*}
    \Bigl( Q_k \left[
      \begin{array}{c}
      \beta_{k-1} + p_{k-1} q_k \\
      d_k + p_k q_k
      \end{array}\right]
    \Bigr)_1 = 0.
  \end{align*}

  \If{$k\ne 2$}
    \State 
      \label{alg:elim:subxsubsub}
    Rotate the subdiagonal and the sub-subdiagonal:
    \begin{align*}
      \gamma_{k-2} \gets \Bigl( Q_k
      \left[
      \begin{array}{c}
        \gamma_{k-2} \\
       -\tilde{q}_k p_{k-2}
      \end{array} \right] \Bigr)_1
    \end{align*}
  \EndIf

  \State 
    \label{alg:elim:diagxsub}
  Rotate the diagonal and the subdiagonal:
  \begin{math}
    \left[
    \begin{array}{c}
      d_{k-1} \\
      \gamma_{k-1}
    \end{array} \right]
    \gets Q_k
    \left[
    \begin{array}{c}
      d_{k-1} \\
      \gamma_{k-1}
    \end{array} \right].
  \end{math}

  \State 
    \label{alg:elim:supxdiag}
  Rotate the superdiagonal and the diagonal:
  \begin{math}
    \left[
    \begin{array}{c}
      \beta_{k-1} \\
      d_{k}
    \end{array} \right]
    \gets Q_k
    \left[
    \begin{array}{c}
      \beta_{k-1} \\
      d_{k}
    \end{array} \right].
  \end{math}

  \State 
    \label{alg:elim:p}
  Rotate $p$:
  \begin{math}
    \left[
    \begin{array}{c}
      p_{k-1} \\
      p_{k}
    \end{array} \right]
    \gets Q_k
    \left[
    \begin{array}{c}
      p_{k-1} \\
      p_{k}
    \end{array} \right]
  \end{math}

  \If{ $\abs{ p_{k-1}q_k }^2 + \abs{ p_k q_k }^2 >
    \abs{\beta_{k-1}}^2 + \abs{d_k}^2$ } 
    \State 
      \label{alg:elim:corr}
    Correct the vector $p$, setting
      \begin{math}
        p_{k-1} \gets -\frac{\beta_{k-1}}{q_k}
      \end{math}
  \EndIf

  \State 
    \label{alg:elim:qtil}
  Rotate $\tilde q$:
  \begin{math}
    \left[
    \begin{array}{c}
      \tilde{q}_{\,k-1} \\
      \tilde{q}_{\,k}
    \end{array} \right]
    \gets Q_k
    \left[
    \begin{array}{c}
      \tilde{q}_{k-1} \\
      \tilde{q}_{k}
    \end{array} \right]
  \end{math}

\EndFor

\State Set $\underline d \gets d$, $\underline \gamma \gets \gamma$, and
$\underline p \gets p$.

\end{algorithmic}
\end{algorithm}

\newpage

\begin{algorithm}[H]
\caption{
  \label{alg:rotback}
(Rotating the matrix back to Hessenberg form) \textbf{Inputs:} This algorithm
accepts as inputs $n-1$ rotation matrices $Q_2, Q_3, \ldots, Q_n \in \C^{n
\times n}$, two vectors $d$ and $\gamma$ representing the diagonal and
subdiagonal, respectively, of an $n\times n$ complex matrix $B$, and two
vectors $\underline{p}$ and $q$ of length $n$, where $B+\underline{p}q^T$ is lower triangular.
\textbf{Outputs:} Letting $U_k \in \C^{n\times n}$, $k=2,3,\ldots,n$, denote
the matrices that rotate the $(k-1,k)$-plane by $Q_k$, this algorithm
returns as its outputs the vectors $\underline d$, $\underline \beta$, and
$\underline q$, where $\underline d$ and $\underline \beta$ represent the
diagonal and superdiagonal, respectively, of the matrix $B U_n^T U_{n-1}^T
\cdots U_2^T$, and $\underline q = U_2 U_3 \cdots U_n q$.}

\begin{algorithmic}[1]

\For{$k=n,n-1,\ldots,2$}

  \State 
    \label{alg:rotback:diagsup}
  Rotate the diagonal and the superdiagonal:
  \begin{align*}
    \left[
    \begin{array}{c}
      d_{k-1} \\
      \beta_{k-1}
    \end{array} \right]
    \gets
    {Q_k}
    \left[
    \begin{array}{c}
      d_{k-1} \\
      -\underline{p}_{k-1} q_k
    \end{array} \right].
  \end{align*}

  \State 
    \label{alg:rotback:subdiag}
  Rotate the subdiagonal and the diagonal:
  \begin{align*}
    d_k \gets \Bigl( {Q_k}
    \left[
    \begin{array}{c}
      \gamma_{k-1} \\
      d_k 
    \end{array} \right] \Bigr)_2
  \end{align*}

  \State 
    \label{alg:rotback:q}
  Rotate $q$:
  \begin{math}
    \left[
    \begin{array}{c}
      q_{k-1} \\
      q_{k}
    \end{array} \right]
    \gets Q_k
    \left[
    \begin{array}{c}
      q_{k-1} \\
      q_{k}
    \end{array} \right]
  \end{math}

\EndFor

\State Set $\underline d \gets d$, $\underline \beta \gets \beta$, and
$\underline q \gets q$.

\end{algorithmic}
\end{algorithm}

\newpage

\begin{algorithm}[H]
\caption{
  \label{alg:qrshift}
(Shifted explicit QR) \textbf{Inputs:} 
This algorithm accepts as inputs two vectors $d$ and $\beta$ representing
the diagonal and superdiagonal, respectively, of an $n\times n$ complex symmetric
matrix $A$, as well as two vectors $p$ and $q$ of length $n$, where $A+pq^T$
is lower Hessenberg. It also accepts a tolerance $\epsilon > 0$, which
determines the accuracy the eigenvalues are computed to.
\textbf{Outputs:} It returns as its output the vector $\lambda$ of length
$n$ containing the eigenvalues of the matrix $A+pq^T$.}

\begin{algorithmic}[1]

\For{$i=1,2,\ldots,n-1$}
  \label{alg:qrshift:outer}
  
  \State Set $\mu_\text{sum} \gets 0$.
  \While{$\beta_i + p_i q_{i+1} \ge \epsilon$}
    \Comment{Check if $(A+pq^T)_{i,i+1}$ is close to zero}

    \State Compute the eigenvalues $\mu_1$ and $\mu_2$ of the $2\times 2$ 
    submatrix
    \begin{math}
      \left[
        \begin{array}{cc}
        d_i + p_i q_i & \beta_i + p_i q_{i+1} \\
        {\beta}_i + p_{i+1} q_i & d_{i+1} + p_{i+1} q_{i+1}
        \end{array}
      \right].
    \end{math}
      \Comment{This is just ${(A+pq^T)_{i:i+1,i:i+1}}$}

    \State Set $\mu$ to whichever of $\mu_1$ and $\mu_2$ is closest
    to $d_i + p_i q_i$.

    \State Set $\mu_\text{sum} \gets \mu_\text{sum} + \mu$.

    \State Set $d_{i:n} \gets d_{i:n} - \mu$.

    \State Perform one iteration of QR (one step of Algorithm~\ref{alg:elim}
    followed by one step of Algorithm~\ref{alg:rotback}) on the
    submatrix $(A+pq^T)_{i:n,i:n}$ defined by the vectors $d_{i:n}$,
    $\beta_{i:n-1}$, $p_{i:n}$, and $q_{i:n}$.
  \EndWhile

  \State Set $d_{i:n} \gets d_{i:n} + \mu_\text{sum}$.

\EndFor

\State Set $\lambda_i \gets d_i + p_i q_i$, for $i=1,2,\ldots,n$.

\end{algorithmic}
\end{algorithm}

\begin{algorithm}[H]
\caption{  \label{alg:nadap_root}
(Non-adaptive rootfinding in square domain $D$) \textbf{Inputs}: This algorithm accepts as inputs an analytic function $f$ on a square domain $D$ centered at $z_0$ with side length $2l$, in which roots of $f$ are to be found, as well as precomputed quantities from Algorithm~\ref{alg:precom}, including vectors $\alpha, \beta$ defining the three-term recurrence, matrices $Q, R$ constituting the reduced QR of the basis matrix $G\in \C^{m\times(n+1)}$, Gaussian points $z_1,z_2,\ldots,z_m$ and Gaussian weights $\tilde{w}_1,\tilde{w}_2,\ldots,\tilde{w}_m$ used in defining $G$. It also accepts two accuracies $\epsilon_{\rm exp}$, the accuracy for the expansion, and $\epsilon_{\rm eig}$, the accuracy for eigenvalues, as well as a positive constant $\delta$ so that roots within the $\delta$-extension of $D$ are included.
 \textbf{Outputs:} It returns as output the roots of $f$ on the domain $D$.}
\begin{algorithmic}[1]
\State Scale and translate $f$ in $D$ to $\tilde{f}$ in $\Omega$: 
\begin{center}
	$\tilde{f}(z) \gets f(lz+z_0)$ for $z\in D$\,.	
\end{center}

\State Form the vector $g$ for least squares:
\label{alg:nonadap:g}

$g \gets \bb{\mat{\sqrt{\tilde{w}_1} \tilde{f}(z_1) & \sqrt{\tilde{w}_2} \tilde{f}(z_2) & \cdots  & \sqrt{\tilde{w}_m} \tilde{f}(z_m)}}^T$.

\State Compute the coefficient vector $c$: \label{alg:nonadap:c}
\begin{center}
$c \gets V \Sigma^{-1} U^{*} g $\,.
\end{center}

\If{ $\abs{c_n}/\norm{c} > \epsilon_{\rm exp}$ } 
\State
Precompute a basis with a larger order $n$ and go to Line~\ref{alg:nonadap:g}.

\EndIf
\State Form vector $q$:
\begin{center}
$q \gets  -\beta_n \bb{\mat{\frac{c_0}{c_n} & \frac{c_1}{c_n} & \cdots & \frac{c_{n-1}}{c_n} } }^T$\,.
\end{center}

\State Perform Algorithm~\ref{alg:qrshift} on the colleague matrix generated by vectors $\alpha$ and $\beta_{1:n-1}$, representing the diagonal and subdiagonal elements of $A$, as well as vectors $e_n$ and $q$ for the rank-1 part. The accuracy of computed eigenvalues $\tilde{r}_1, \tilde{r}_2,\ldots,\tilde{r}_n$ is determined by $\epsilon_{\rm eig}$.

\State
Keep eigenvalues $\tilde{r}_i$ if
\begin{math}
\abs{\mathrm{Re}\,\tilde{r}_i} < 1+\delta \quad  \mathrm{and}\quad \abs{\mathrm{Im}\,\tilde{r}_i} < 1+\delta.	
	\end{math}
\State Scale and translate the remaining eigenvalues:
\begin{center}
	$r_i\gets l \tilde{r}_i + z_0$\,.
\end{center}
\State
	Return all $r_i$ as roots of $f$ on the $\delta$-extended domain of $D$.

\end{algorithmic}
\end{algorithm}

\begin{algorithm}[H]
\caption{  \label{alg:adap_root}
(Adaptive rootfinding in square domain $D$) \textbf{Inputs}: This algorithm accepts as inputs an analytic function $f$ on a square domain $D$ centered at $z_0$ with side length $2l$, in which roots of $f$ are to be found, as well as precomputed quantities from Algorithm~\ref{alg:precom}, including vectors $\alpha, \beta$ defining the three-term recurrence, matrices $Q, R$ constituting the reduced QR of the basis matrix $G\in \C^{m\times(n+1)}$, Gaussian points $z_1,z_2,\ldots,z_m$ and Gaussian weights $\tilde{w}_1,\tilde{w}_2,\ldots,\tilde{w}_m$ used in defining $G$. It also accepts two accuracies $\epsilon_{\rm exp}$, the accuracy for the expansion, and $\epsilon_{\rm eig}$, the accuracy for eigenvalues, as well as $n$ the order of expansion and 
a positive constant $\delta$ so that roots within the $\delta$-extension of $D$ are included.
 \textbf{Outputs:} It returns as output the roots of $f$ in the domain $D$.}
 
\begin{algorithmic}[1]
\State Set $D_0 \gets D$, $l^{(0)} \gets l$ and  $z^{(0)}_0 \gets z_0$.
\State Set $i\gets 0$. \Comment{Square index}
\State Set $k\gets 1$. \Comment{Total number of squares}
\While{ $i < k$ } \Comment{If the last square has not been reached}
\State Scale and translate $f$ on $D_i$ to $\tilde{f}$ on $\Omega$: 
\begin{center}
	$\tilde{f}(z) \gets f(l^{(i)} z+z^{(i)}_0)$ for $z\in D_i$\,.	
\end{center}

\State Form the vector $g^{(i)}$ for least squares:

$g^{(i)} \gets \bb{\mat{\sqrt{\tilde{w}_1} \tilde{f}(z_1) & \sqrt{\tilde{w}_2} \tilde{f}(z_2) & \cdots  & \sqrt{\tilde{w}_m} \tilde{f}(z_m)}}^T$

\State Compute the coefficient vector $c^{(i)} $: 
\begin{center}
$c^{(i)}  \gets R^{-1}Q^* g^{(i)} $\,.
\end{center}

\If{ $\abs{c^{(i)}_{n}}/\norm{c^{(i)}} > \epsilon_{\rm exp}$ }  \Comment{Keep dividing if not convergent}
\State Divide $D_i$ into four smaller squares $D_{k+1},D_{k+2},D_{k+3},D_{k+4}$. 
 \State Set $k\gets k +4$.
\Else					
\Comment{Keep the coefficient vector if convergent}

\State Form vector $q^{(i)}$:
\begin{center}
$q^{(i)} \gets  -\beta_n \bb{\mat{\frac{c_0}{c_n} & \frac{c_1}{c_n} & \cdots & \frac{c_{n-1}}{c_n} } }^T$\,.
\end{center}

\EndIf
\State Set $i \gets i+1$. \Comment{Go to the next square}
\EndWhile

\For{$D_i$ with a convergent expansion}
\State Perform Algorithm~\ref{alg:qrshift} on the colleague matrix generated by vectors $\alpha$ and $\beta_{1:n-1}$, representing the diagonal and subdiagonal elements of $A$, as well as vectors $e_n$ and $q^{(i)}$ for the rank-1 part. The accuracy of computed eigenvalues $\tilde{r}^{(i)}_1, \tilde{r}^{(i)}_2,\ldots,\tilde{r}^{(i)}_n$ is determined by $\epsilon_{\rm eig}$.

\State
Keep eigenvalue $\tilde{r}^{(i)}_j$ if
\begin{math}
\abs{\mathrm{Re}\,\tilde{r}^{(i)}_j} < 1+\delta \quad  \mathrm{and}\quad \abs{\mathrm{Im}\,\tilde{r}^{(i)}_j} < 1+\delta.	
	\end{math}
\State Scale and translate the remaining eigenvalues:
\begin{center}
	$r^{(i)}_j \gets l^{(i)} \tilde{r}^{(i)}_j + z^{(i)}_0$\,.
\end{center}
\EndFor

\State
	Collect all roots $r^{(i)}_j$ and remove duplicated roots from neighboring squares. Return all the remaining roots as the roots of $f$ in the domain $D$.

\end{algorithmic}
\end{algorithm}

%
%
%


\newpage
\section{Numerical results\label{sec:result}}

In this section, we demonstrate the performance of Algorithm~\ref{alg:nadap_root} (non-adaptive) and Algorithm~\ref{alg:adap_root} (adaptive) for finding all roots of analytic functions over any specified square domain $D$. Because the use of complex orthogonal matrix, no proof of backward stability is provided in this paper. Numerical experiments show our complex orthogonal QR algorithm is stable in practice and behaves similarly to the version with unitary rotations in \cite{serkh2021}.

We apply Algorithm~\ref{alg:nadap_root} to the first three examples and Algorithm~\ref{alg:adap_root} to the remaining two. In all experiments in this paper, we do not use Newton's method to refine roots computed by algorithms, although the refinement can be done at little cost to achieve full machine accuracy. We set the two accuracies $\epsilon_{\rm exp}$ (for the polynomial expansion) and $\epsilon_{\rm eig}$ (for the QR algorithm) both to the machine epsilon.
We set the domain extension parameter $\delta$ to be $10^{-6}$, determining the relative extension of the square within which roots are kept. For all experiments, we used the same set of polynomials as the basis, precomputed to order $n=100$, with $60$ Gaussian nodes  on each side of the square (so in total $m=240$). The condition number of this basis (See Section~\ref{sec:data}) is about 1000.
 When different expansion orders are used, the number of points on each side is always kept as $60$. Given a function $f$ whose roots are to be found, we measure the accuracy of the computed roots $\hat{z}$ given by the quantity $\eta(\hat{z})$ defined by
\begin{equation}
	\eta (\hat{z}) \eqdef \abs{\frac{f(\hat{z})}{f'(\hat{z})}}\,,
\end{equation}
where $f'$ is the derivative of $f$; it is the size of the Newton step if we were to refine the computed root $\hat{z}$ by Newton's method.

In all experiments, we report the order $n$ of the approximating polynomial, the number $n_{\rm roots} $ of computed roots, the value of $\mathrm{max}_i ~\eta(\hat{z}_i)$, where the maximum is taken over all roots found over the region of interest.
For the non-adaptive algorithm, we report the norm $\norm{q}$ of the vector $q$ in the colleague matrix (See (\ref{eq:qvec})) and the size of rotations in QR iterations. 
For the adaptive version, we report the maximum level $n_{\rm levels}$ of the recursive division and the number of times $n_{\rm eig}$ eigenvalues are computed by Algorithm~\ref{alg:qrshift}. The number $n_{\rm eigs}$ is also the number of squares formed by the recursive division.

We implemented the algorithms in FORTRAN 77, and compiled it using GNU Fortran, version 12.2.0. For all timing experiments, the Fortran codes were compiled with the optimization \texttt{-O3} flag. All experiments are conducted on a MacBook Pro laptop, with 64GB of RAM and an Apple M1 Max CPU.



\subsection{$f_{\rm cosh}$: A function with a pole near the square}
Here we consider a function given by
\begin{equation}
	f_{\rm cosh}(z) = \frac{\cosh(3\pi z/2)}{z-2},
\end{equation}
over a square centered at $0$ with side length $2$, with a singularity outside of the square domain. Due to the singularity at $z=2$, the sizes of the approximating polynomial coefficients decay geometrically, as shown in Figure~\ref{fig:conv}. The results of numerical experiments with order $n=80$ and $100$ are shown in Table~\ref{tab:ex0_double}.

\begin{table}[htb]
\centering
\begin{tabular}{l |l|l| l | l }
$n$ & $\norm{q}$  & $n_{\rm roots}$ & $\mathrm{max}_i \eta(\hat{z}_i)$ & Timings (s) \\
\hline
80 & $0.11\cdot 10^{17}$ &4 & $0.55\cdot 10^{-11}$ & $0.27\cdot 10^{-1}$\\
100 & $0.18\cdot 10^{17}$ &4 & $0.83\cdot 10^{-11}$ & $0.37\cdot 10^{-1}$
\end{tabular}
\caption{The results of computing roots of $f_{\rm cosh}(z)$ in double precision, using the non-adaptive Algorithm~\ref{alg:nadap_root} with the polynomial expansion order $n=80$ and $n=100$. }
\label{tab:ex0_double}
\end{table}

\begin{figure}[h]
	\centering
	\begin{subfigure}[t]{0.47\textwidth}
		\centering
		\includegraphics[width=\textwidth]{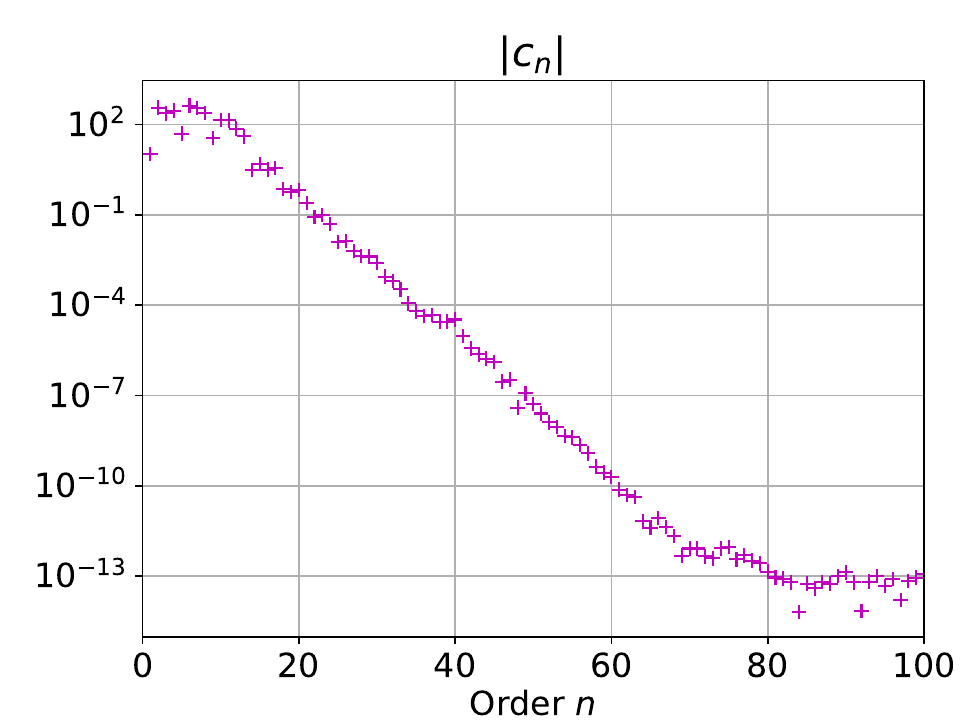}
				\caption{}
	\end{subfigure}	
	~
	\begin{subfigure}[t]{0.48\textwidth}
		\centering
		\includegraphics[width=\textwidth]{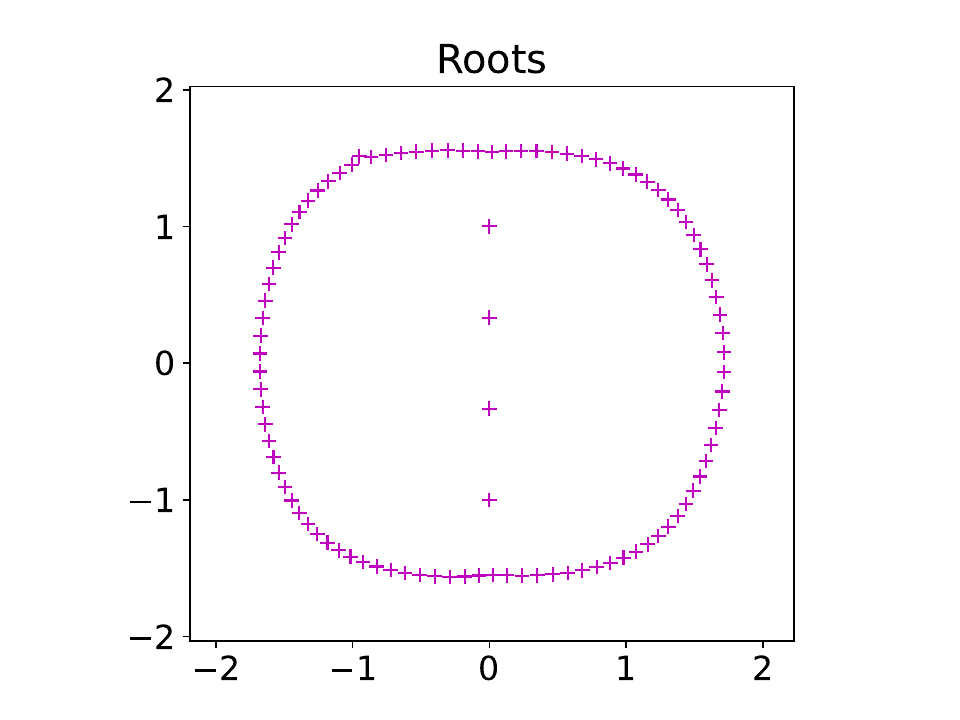}
				\caption{}
	\end{subfigure}	
	\caption{The magnitude of the leading expansion coefficients $\abs{c_{n}}$ of $f_{\rm cosh}$ is plotted in (a). Due to the singularity at $z=2$, the coefficients decays geometrically. All computed roots of the approximating polynomial of $f_{\cosh}$ are shown in (b).
	\label{fig:conv}}
\end{figure}

\subsection{$f_{\rm poly}$: A polynomial with simple roots \label{sec:ex_poly}}
Here we consider a polynomial of order $5$, whose formula is given by
\begin{equation}
	f_{\rm poly}(z) = (z-0.5)(z-0.9)(z+0.8)(z-0.7\I)(z+0.1\I).
\end{equation}
All its roots are simple and inside a square of side length $l=2$ centered at $z_0=0$. We construct a polynomial of order $n$ and apply Algorithm~\ref{alg:nadap_root} to find its roots. The results of our experiment are shown in Table~\ref{tab:ex1_double} (double precision) and Table~\ref{tab:ex1_ext} (extended precision). 
In Figure~\ref{fig:poly}, the size of complex orthogonal rotations for $n=40$ in the QR iterations for finding roots of $f_{\rm poly}$
are provided.
It is clear that the error in computed roots are insensitive to the size of $\norm{q}$ and the order of polynomial approximation used. This is the feature of the unitary version of our QR algorithms in \cite{serkh2021}, which is proven to be structured backward stable.
Although no similar proof is provided in this paper, results in Table~\ref{tab:ex1_double} and \ref{tab:ex1_ext} strongly suggest our algorithm has similar stability properties as the unitary version in \cite{serkh2021}. 

\begin{figure}[H]
	\centering

	\begin{subfigure}[t]{0.48\textwidth}
		\centering
		\includegraphics[width=\textwidth]{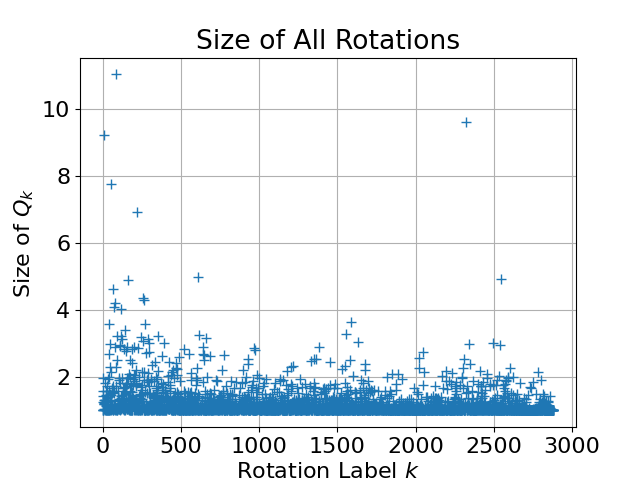}
		\caption{}
		\label{fig:poly1}
	\end{subfigure}	
	~
	\begin{subfigure}[t]{0.48\textwidth}
		\centering
		\includegraphics[width=\textwidth]{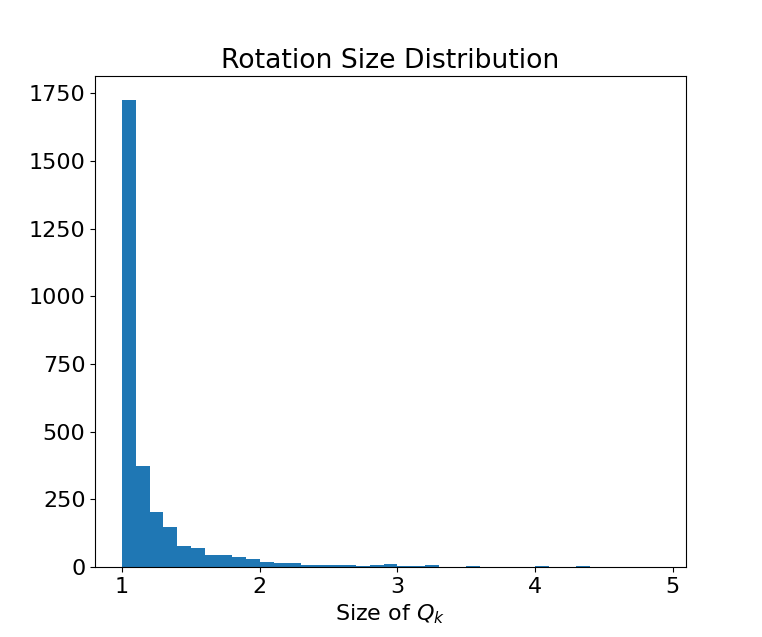}
		\caption{}
		\label{fig:poly2}
	\end{subfigure}
	\caption{The size of all 2882 complex orthogonal rotations defined by $\sqrt{\abs{c}^2+\abs{s}^2}$, in the QR iterations ($n=40$) for finding roots of $f_{\rm poly}$ is shown in (a). The distribution of rotations of size smaller than 5 is shown in (b). \label{fig:poly}}
\end{figure}

\begin{table}[htb]
\centering
\begin{tabular}{l |l|l| l | l}
$n$ & $\norm{q}$ & $n_{\rm roots}$ &$\mathrm{max}_i \eta(\hat{z}_i)$ & Timings (s) \\
\hline
5 & $0.52\cdot 10^{1}$ &5 & $0.10\cdot 10^{-12}$ & $0.53\cdot 10^{-3}$\\
6 &$0.20\cdot 10^{17}$ &5 & $0.25\cdot 10^{-13}$ & $0.53\cdot 10^{-3}$\\
50 &$0.10\cdot 10^{16}$ &5 & $0.19\cdot 10^{-13}$ & $0.15\cdot 10^{-1}$\\
100 &$0.15\cdot 10^{17}$ &5 & $0.64\cdot 10^{-13}$ & $0.37\cdot 10^{-1}$

\end{tabular}
\caption{The results of computing roots of $f_{\rm poly}$, a polynomial of order 5, in double precision, using Algorithm~\ref{alg:nadap_root} with the polynomial expansion order $n=5,6,50$ and $100$.}
\label{tab:ex1_double}
\end{table}

\begin{table}[htb]
\centering
\begin{tabular}{l |l|l| l | l }
$n$ & $\norm{q}$ & $n_{\rm roots}$ &$\mathrm{max}_i \eta(\hat{z}_i)$ & Timings (s) \\
\hline
5 & $0.52\cdot 10^{1}$ &5 & $0.90\cdot 10^{-30}$ & $0.12\cdot 10^{-1}$\\
6 &$0.61\cdot 10^{34}$ &5 & $0.78\cdot 10^{-30}$ & $0.12\cdot 10^{-1}$\\
50 &$0.76\cdot 10^{33}$ &5 & $0.94\cdot 10^{-30}$ & $0.85\cdot 10^{-1}$\\
100 &$0.66\cdot 10^{34}$ &5 & $0.94\cdot 10^{-30}$ & $0.27\cdot 10^{0} $

\end{tabular}
\caption{The results of computing roots of $f_{\rm poly}(z)$, a polynomial of order 5, in extended precision, using Algorithm~\ref{alg:nadap_root} with the polynomial expansion order $n=5,6,50$ and $100$.}
\label{tab:ex1_ext}
\end{table}

\subsection{$f_{\rm mult}$: A polynomial with multiple roots \label{sec:ex_mult_poly}}
Here we consider a polynomial of order $12$ that has identical roots as $f_{\rm poly}$ above with increased multiplicity:
\begin{equation}
	f_{\mathrm{mult}} (z) = (z-0.5)^5(z-0.9)^3(z+0.8)(z-0.7\I)(z+0.1\I)^2.
\end{equation}
We construct a polynomial of order $n=30$ and apply Algorithm~\ref{alg:nadap_root} to find its roots in both double and extended precision. All computed roots and their error estimation are shown in Figure~\ref{fig:mult_poly}. 
The error is roughly proportional to $u^{1/m}$, where $u$ is the machine epsilon and $m$ is the root's corresponding multiplicity. 
 As discussed in Remark~\ref{rmk:mult}, this can be used to distinguish multiple roots from redundant roots when removing redundant roots in Algorithm~\ref{alg:adap_root}.

\begin{figure}[H]
	\centering

	\begin{subfigure}[t]{0.55 \textwidth}
		\centering
		\includegraphics[width=\textwidth]{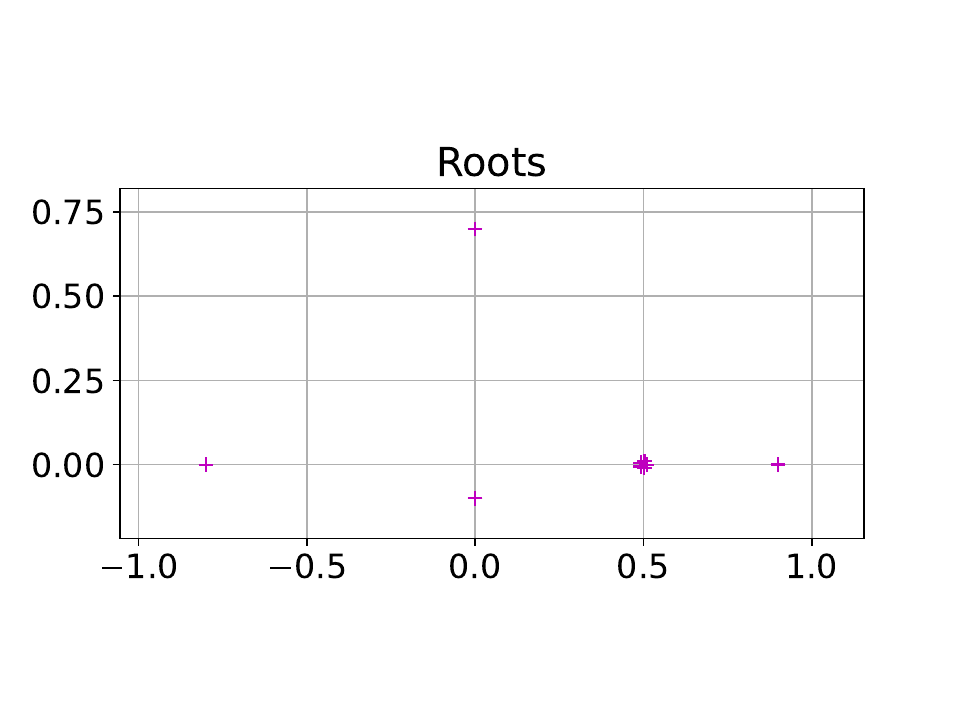}
		\caption{}
		\label{fig:mult_roots}
	\end{subfigure}	

	\begin{subfigure}[t]{0.43\textwidth}
		\centering
		\includegraphics[width=\textwidth]{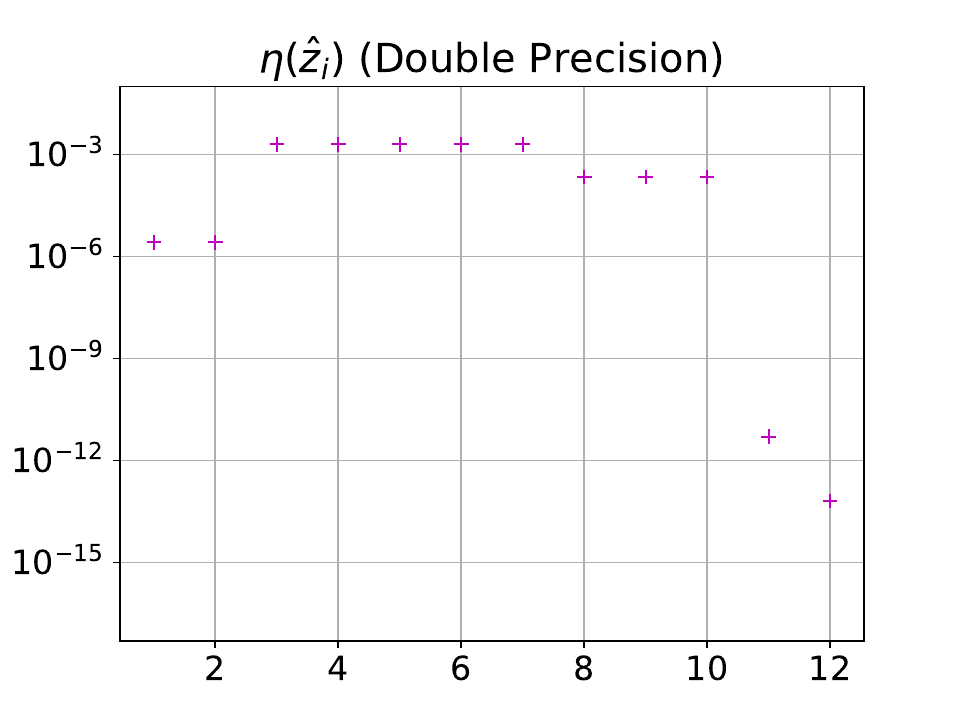}
		\caption{}
		\label{fig:mult_error}
	\end{subfigure}
	~	
	\begin{subfigure}[t]{0.43\textwidth}
		\centering
		\includegraphics[width=\textwidth]{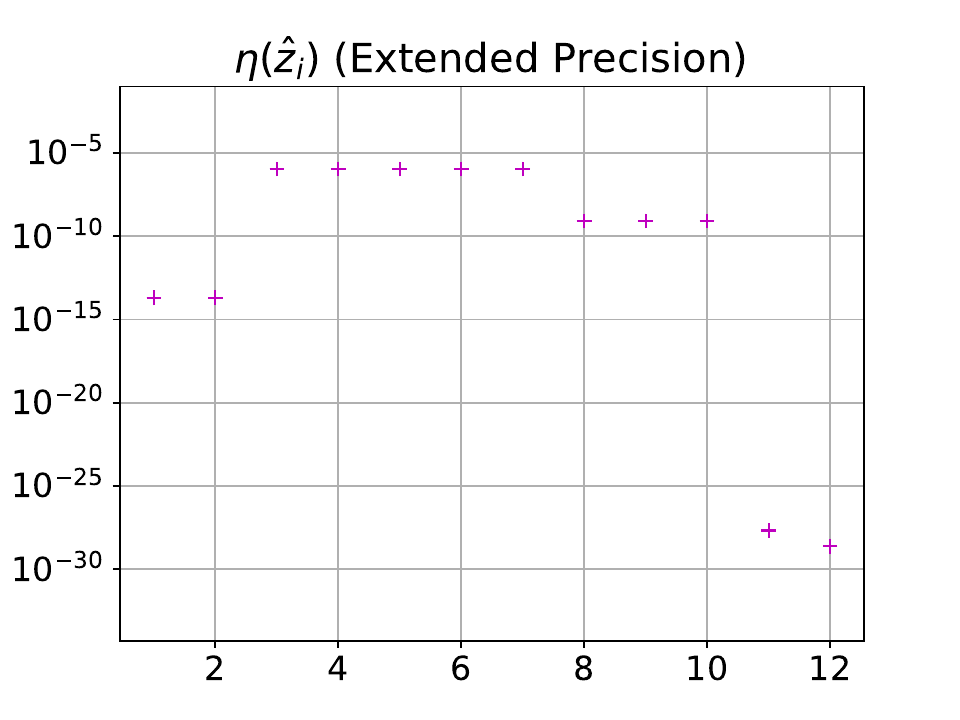}
		\caption{}
		\label{fig:mult_error_ext}
	\end{subfigure}
	\caption{ The roots of $f_{\mathrm{ mult}}$ computed by Algorithm~\ref{alg:nadap_root} in double precision are shown in (a), and their errors in double and extended precision computation are shown in (b) and (c) respectively. The error is roughly proportional to $u^{1/m}$, where $u$ is the machine epsilon and $m$ is the multiplicity of the roots.
	\label{fig:mult_poly}}
\end{figure}

\subsection{$f_{\rm clust}$: A function with clustering zeros\label{sec:cluster}}
Here we consider a function given by
\begin{equation}
	f_{\rm clust}(z) = \sin(\frac{100}{\Exp{\I\pi/4} z -2})\,,
\end{equation}
which has roots clustering around the singularity $z_{\star}=\sqrt{2} -\I\sqrt{2}$ along the line $\theta=-{\pi}/{4}$\,.
We apply Algorithm~\ref{alg:adap_root} to find roots of $f_{\rm clust}$ within a square of side length $2.75$ centered at the origin $z_0=0$.
Near the singularity, the adaptivity of Algorithm~\ref{alg:nadap_root} allows the size of squares to be chosen automatically, so that the polynomial expansions with a fixed order $n$ can resolve the fast oscillation of $f_{\rm clust}$ near the singularity $z_{\star}$.

The results of our numerical experiment are shown in Table~\ref{tab:ex3_double} (in double precision) and Table~\ref{tab:ex3_ext} (in extended precision).
Figure~\ref{fig:cluster} contains the roots $\hat{z}_i$ of $f_{\rm clust}$ found, as well as  every center of squares formed by the recursive subdivision and the error estimation $\eta(\hat{z}_i)$ of roots $\hat{z}_i$.
It is worth noting that the roots found by order $n=45$ expansions are consistently less accurate than those found by order $n=30$ ones, as indicated by error estimation shown in Figure~\ref{fig:cluster4} and \ref{fig:cluster5}. This can be explained by the following observation. The larger the expansion order $n$, the larger the square on which the expansion converges, as shown in Figure~\ref{fig:cluster2} and \ref{fig:cluster3}. On a larger square, the maximum magnitude of the function $f$ tends to be larger as $f$ is analytic.  Since the maximum principle (Theorem~\ref{thm:max}) only provides bounds for the pointwise \emph{absolute} error $\mathrm{max}_{z\in\partial D}\abs{f(z)-p(z)}$ for any approximating polynomial $p$, larger $f$ values on the boundary leads larger errors of $p$ inside. This eventually leads to less accurate roots when a larger order $n$ is adopted.

Although a larger order expansion leads to less accurate roots, increasing $n$ indeed leads to significantly fewer divided squares, as demonstrated by the numbers of eigenvalue problems $n_{\rm eigs}$ shown in Table~\ref{tab:ex3_double} and \ref{tab:ex3_ext}. In practice, one can choose a relatively large expansion order $n$, so that fewer eigenvalue problems need to be solved while the computed roots are still reasonably accurate, achieving a balance between efficiency and robustness. 
Once all roots are located within reasonable accuracy, one Newton refinement will give full machine accuracy to the computed roots.

\begin{figure}
	\centering

	\begin{subfigure}[t]{0.6\textwidth}
		\centering
		\includegraphics[width=\textwidth]{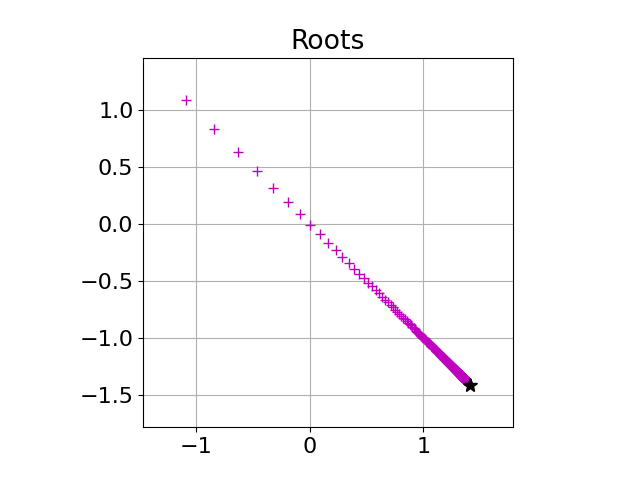}
		\caption{}
		\label{fig:cluster1}
	\end{subfigure}	
	
	\begin{subfigure}[t]{0.48\textwidth}
		\centering
		\includegraphics[width=\textwidth]{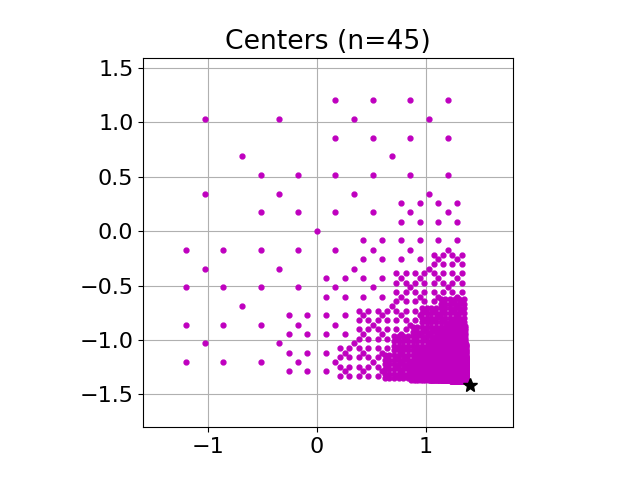}
		\caption{}
		\label{fig:cluster2}
	\end{subfigure}
	~
	\begin{subfigure}[t]{0.48\textwidth}
		\centering
		\includegraphics[width=\textwidth]{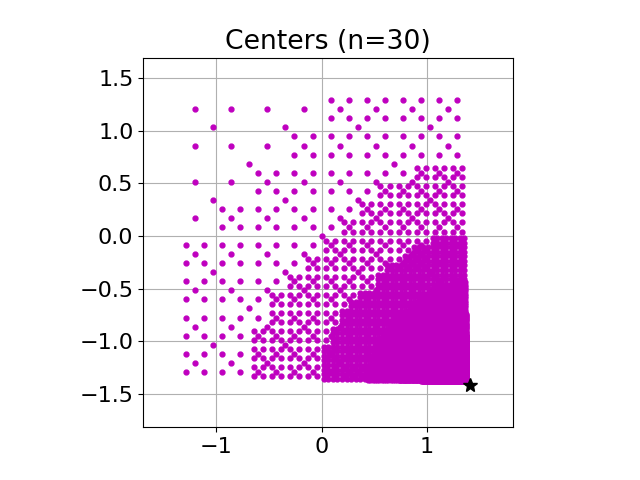}
		\caption{}
		\label{fig:cluster3}
	\end{subfigure}
	
		\centering
	\begin{subfigure}[t]{0.45\textwidth}
		\centering
		\includegraphics[width=\textwidth]{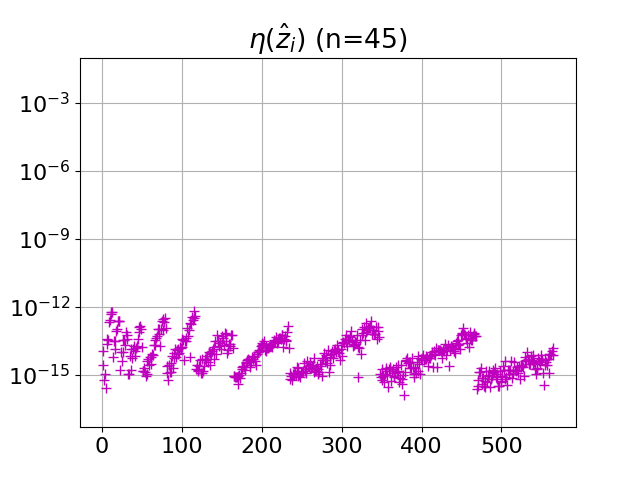}
		\caption{}
		\label{fig:cluster4}
	\end{subfigure}
	~
	\begin{subfigure}[t]{0.45\textwidth}
		\centering
		\includegraphics[width=\textwidth]{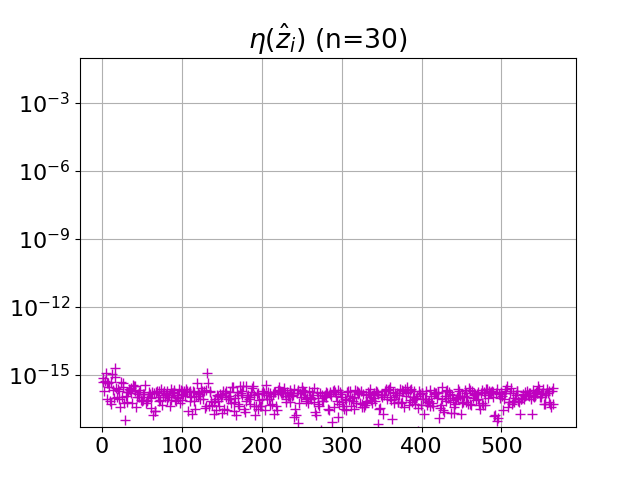}
		\caption{}
		\label{fig:cluster5}
	\end{subfigure}
	\caption{All roots $\hat{z}_i$ of $f_{\rm clust}$ within a square of side length $2.75$ centered at $z_0=0$ found by Algorithm~\ref{alg:adap_root} with polynomials of order $n=30$ are shown in (a). The centers of all squares during the recursive division with polynomials of order $n=45$ and $n=30$ are respectively shown in (b) and (c). The error estimations $\eta(\hat{z}_i)$ of roots $\hat{z}_i$ with polynomials of order $n=45$ and $n=30$ are shown in (d) and (e). The star $\star$ in (a)-(b) indicates the singularity location at $z_{\star}=\sqrt{2}-\I\sqrt{2}$.	\label{fig:cluster}}
\end{figure}

\begin{table}[htb]
\centering
\begin{tabular}{l |l|l| l | l | l}
$n$ & $n_{\rm roots}$ &$\mathrm{max}_i \eta(\hat{z}_i)$ & $n_{\rm levels}$ & $n_{\rm eigs}$ & Timings (s) \\
\hline
45 & 565 & $0.68\cdot 10^{-12}$ & 14 & 8836 & $0.30\cdot 10^1$\\
30 & 565 & $0.19\cdot 10^{-14}$ & 16 & 76864 & $0.14\cdot 10^2$
\end{tabular}
\caption{The results of computing roots of $f_{\rm clust}(z)$ in double precision, using the adaptive Algorithm~\ref{alg:adap_root} with the polynomial expansion order $n=45$ and $n=30$. }
\label{tab:ex3_double}
\end{table}

\begin{table}[h]
\centering
\begin{tabular}{l |l|l| l | l | l }
$n$ & $n_{\rm roots}$ &$\mathrm{max}_i \eta(\hat{z}_i)$ & $n_{\rm levels}$ & $n_{\rm eigs}$  & Timings (s)  \\
\hline
100 & 565 & $0.30\cdot 10^{-23}$ & 13 & 1852 & $0.34\cdot 10^3$\\
70 & 565 & $0.83\cdot 10^{-29}$ & 14 & 8899  & $0.88\cdot 10^3$
\end{tabular}
\caption{The results of computing roots of $f_{\rm clust}(z)$ in extended precision, using the adaptive Algorithm~\ref{alg:adap_root} with the polynomial expansion order $n=100$ and $n=70$. }
\label{tab:ex3_ext}
\end{table}

\subsection{$f_{\rm entire}$: An entire function}
Here we consider a function given by
\begin{equation}
	f_{\rm entire}(z) = \frac{\sin(3\pi z)}{ z-2}\,.
\end{equation}
Since $z=2$ is a simple zero of the numerator $\sin(3\pi z)$, the function $f_{\rm entire}$ is an entire function. This function only has simple roots on the real line. We apply Algorithm~\ref{alg:adap_root} to find roots of $f_{\rm entire}$ within a square of side length $50$ centered at $z_0=10-20\I$. Although our algorithms are not designed for such an environment, Algorithm~\ref{alg:adap_root} is still effective in finding the roots. Figure~\ref{fig:entire} contains the roots $\hat{z}_i$ of $f$ found, every center of squares formed by the recursive division and the error estimation $\eta(\hat{z}_i)$ of root $\hat{z}_i$. The scale of function $f_{\rm entire}$ does not vary much across the whole region of interests, so the algorithm divides the square uniformly. 

The results of our numerical experiment are shown in Table~\ref{tab:ex4_double} (in double precision) and Table~\ref{tab:ex4_ext} (in extended precision). As can be seen from the tables, the larger the expansion order used, the larger the square on which expansions converge, thus the lower the precision of roots computed for the same reason explained in Section~\ref{sec:cluster}. The most economic approach is to combine a reasonably large expansion order $n$ and a Newton refinement in the end.

\begin{figure}[H]
	\centering

	\begin{subfigure}[t]{0.65\textwidth}
		\centering
		\includegraphics[width=\textwidth]{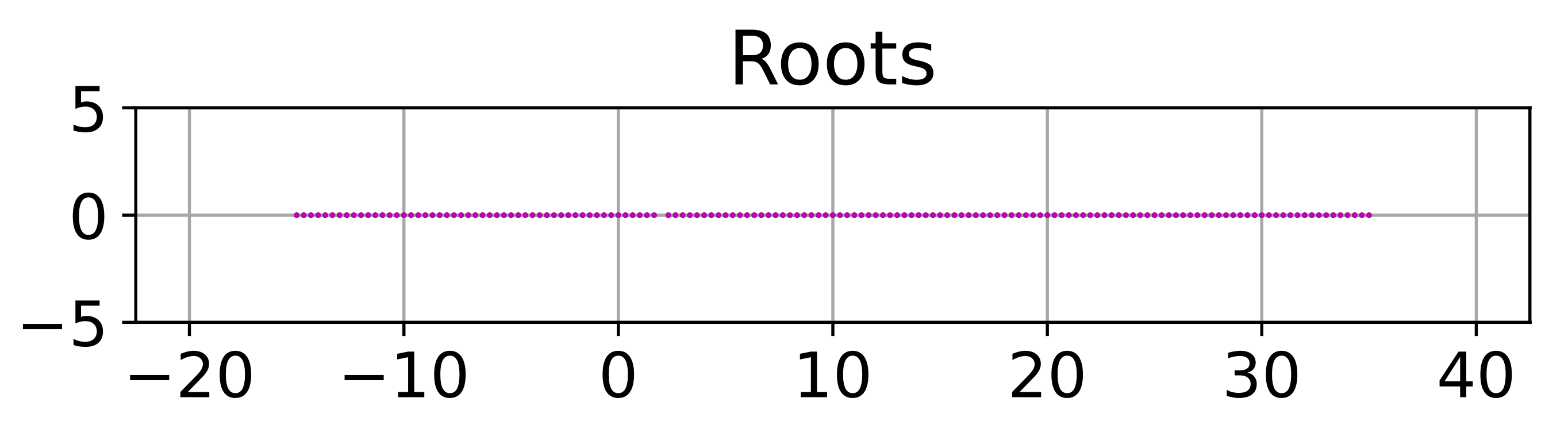}
		\caption{}
		\label{fig:entire1}
	\end{subfigure}	
	
	\begin{subfigure}[t]{0.48\textwidth}
		\centering
		\includegraphics[width=\textwidth]{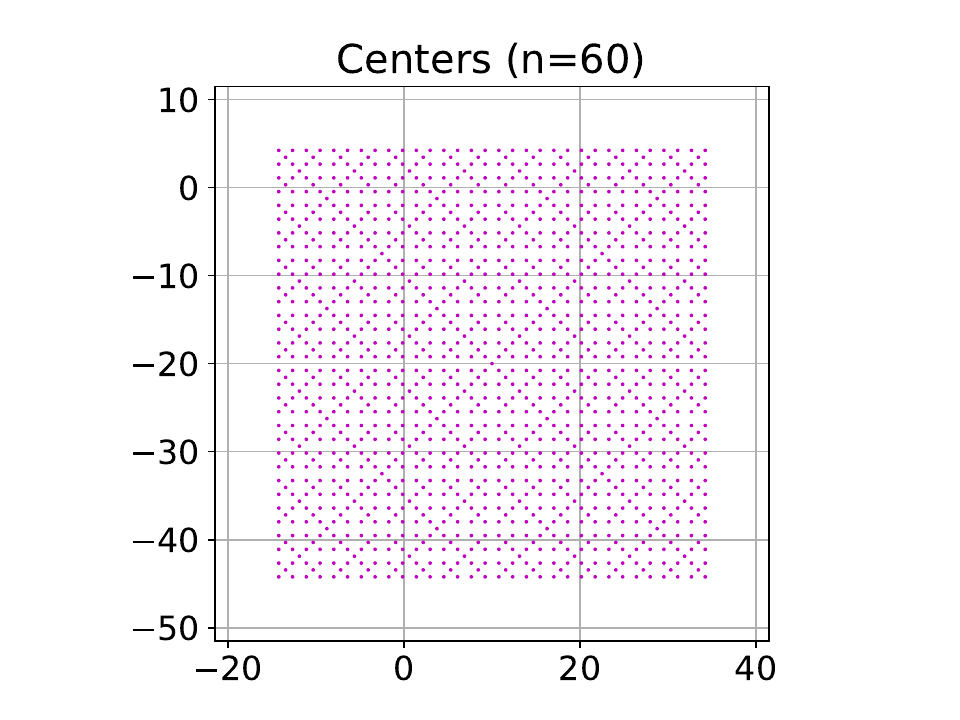}
		\caption{}
		\label{fig:entire2}
	\end{subfigure}
	~
	\begin{subfigure}[t]{0.48\textwidth}
		\centering
		\includegraphics[width=\textwidth]{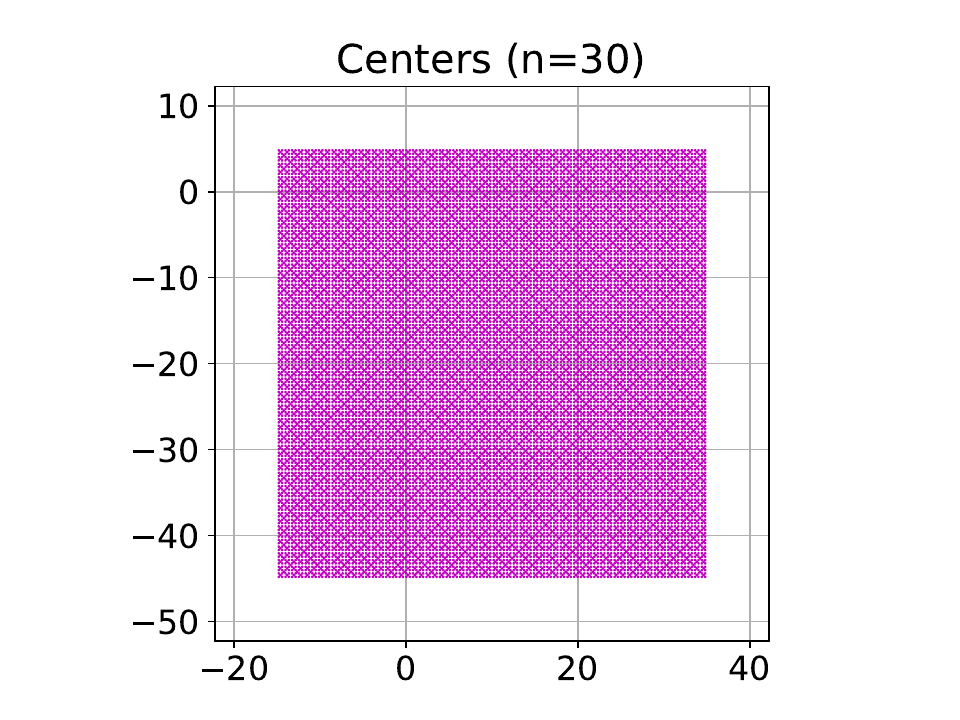}
		\caption{}
		\label{fig:entire3}
	\end{subfigure}
	
		\centering
	\begin{subfigure}[t]{0.45\textwidth}
		\centering
		\includegraphics[width=\textwidth]{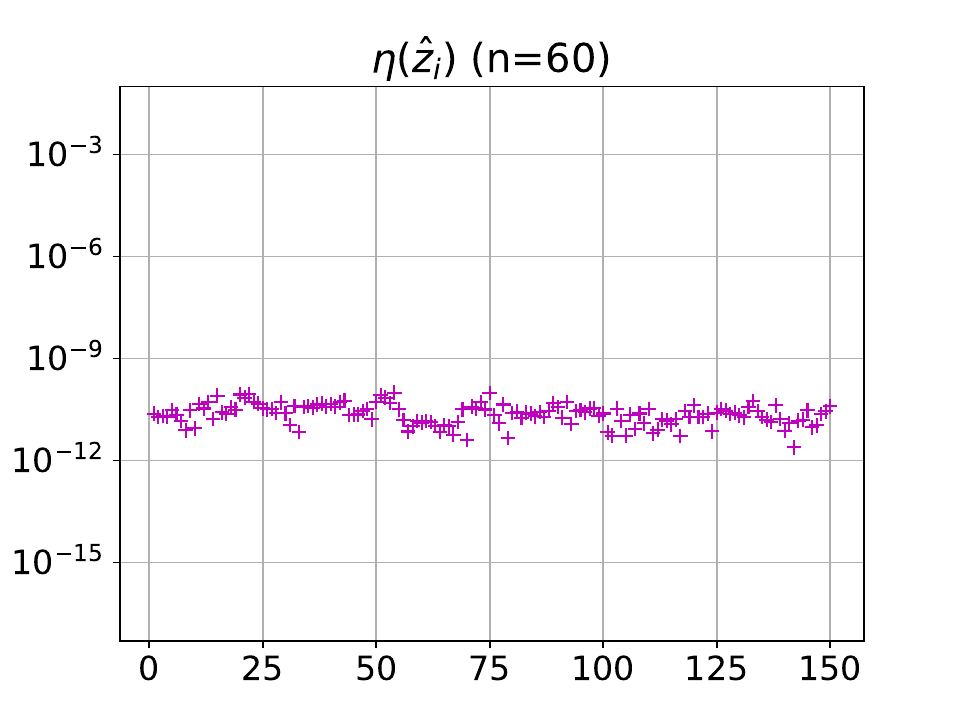}
		\caption{}
		\label{fig:entire4}
	\end{subfigure}
	~
	\begin{subfigure}[t]{0.45\textwidth}
		\centering
		\includegraphics[width=\textwidth]{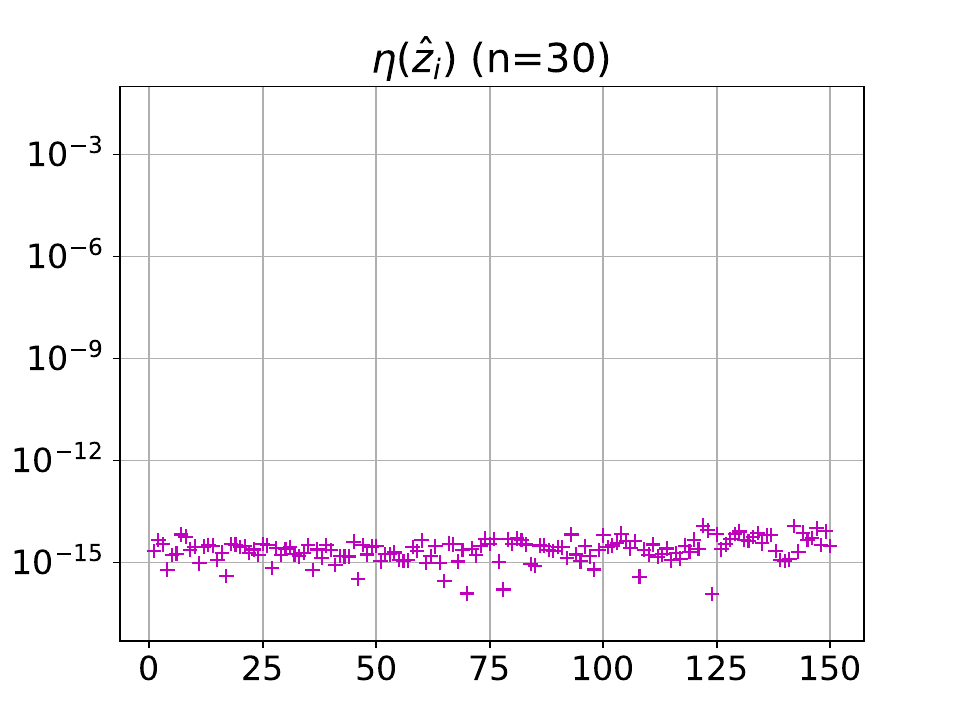}
		\caption{}
		\label{fig:entire5}
	\end{subfigure}
	\caption{All roots $\hat{z}_i$ of $f_{\rm entire}$ within a square of side length $50$ centered at $z_0=10-20\I$ found by Algorithm~\ref{alg:adap_root} with polynomials of order $n=30$ are shown in (a). The centers of all squares during the recursive division with polynomials of order $n=60$ and $n=30$ are respectively shown in (b) and (c). The error estimations $\eta(\hat{z}_i)$ of roots $\hat{z}_i$ with polynomials of order $n=60$ and $n=30$ are shown in (d) and (e).	\label{fig:entire}}
\end{figure}

\begin{table}[h]
\centering
\begin{tabular}{l |l|l| l | l | l }
$n$ & $n_{\rm roots}$ &$\mathrm{max}_i \eta(\hat{z}_i)$ & $n_{\rm levels}$ & $n_{\rm eigs}$  & Timings (s)  \\
\hline
60 & 150 & $0.99\cdot 10^{-10}$ & 6 & 1024 & $0.56\cdot 10^0$\\
30 & 150 & $0.22\cdot 10^{-13}$ & 8 & 16384  & $0.31\cdot 10^1$
\end{tabular}
\caption{The results of computing roots of $f_{\rm entire}(z)$ in double precision, using the adaptive Algorithm~\ref{alg:adap_root} with the polynomial expansion order $n=60$ and $n=30$. }
\label{tab:ex4_double}
\end{table}

\begin{table}[h]
\centering
\begin{tabular}{l |l|l| l | l | l }
$n$ & $n_{\rm roots}$ &$\mathrm{max}_i \eta(\hat{z}_i)$ & $n_{\rm levels}$ & $n_{\rm eigs}$  & Timings (s)  \\
\hline
100 & 150 & $0.14\cdot 10^{-23}$ & 5 & 256 & $0.46\cdot 10^2$\\
70 & 150 & $0.35\cdot 10^{-26}$ & 6 & 1024  & $0.99\cdot 10^2$\\
65 & 150 & $0.30\cdot 10^{-29}$ & 7 & 4096  & $0.36\cdot 10^3$
\end{tabular}
\caption{The results of computing roots of $f_{\rm entire}$ in extended precision, using the adaptive Algorithm~\ref{alg:adap_root} with the polynomial expansion order $n=100,70$ and $65$. }
\label{tab:ex4_ext}
\end{table}


\section{Conclusions\label{sec:conclude}}
In this paper, we described a method for finding all roots of analytic functions in square domains in the complex plane, which can be  viewed as a generalization of rootfinding by classical colleague matrices on the interval. This approach relies on the observation that complex orthogonalizations, defined by complex inner product with random weights, produce
reasonably well-conditioned bases satisfying three-term recurrences in compact simply-connected domains.
This observation is relatively insensitive to the shape of the domain and locations of points chosen for construction. We demonstrated by numerical experiments the condition numbers of constructed bases scale almost linearly with the order of the basis. 
When such a basis is constructed on a square domain, all roots of polynomials expanded in this basis are found as eigenvalues of a class of generalized colleague matrices. As a result, all roots of an analytic function over a square domain are found by finding those of a proxy polynomial that approximates the function to a pre-determined accuracy. Such a generalized colleague matrix is lower Hessenberg and consist of a complex symmetric part with a rank-1 update. Based on this structure, a complex orthogonal QR algorithm, which is a straightforward extension of the one in \cite{serkh2021}, is introduced for computing  eigenvalues of generalized colleague matrices. The complex orthogonal QR also takes $O(n^2)$ operations to find all roots of a polynomial of order $n$
and exhibits structured backward stability, as demonstrated by numerical experiments. 

Since the method we described is not limited to square domains, this work can be easily generalized to rootfindings over general compact complex domains. Then the corresponding colleague matrices will have structures identical to the one in this paper, so our complex orthogonal QR can be applied similarly. 
However, in some cases, when the conformal map is analytic and can be computed accurately, it is easier to map the rootfinding problem to the unit disk from the original domain, so that companion matrices can be constructed and eigenvalues are computed by the algorithm in \cite{aurentz2018fast}.

There are several problems in this paper that require further study. 
First, the observation that complex inner products defined by random complex weights lead to well-conditioned polynomial basis is not understood. Second, the numerical stability of our QR algorithm is not proved due to the use of complex orthogonal rotations.
Numerical evidence about mentioned problems have been provided in this paper and they are under vigorous investigation. Their analysis and proofs will appear in future works. 

\section*{Acknowledgement}
The authors would like to thank Kirill Serkh whose ideas form the foundation of this paper.
\vspace{5mm}

\newpage

\bibliographystyle{abbrv}
\bibliography{root} 


\end{document}